\documentclass[11pt]{article}
\usepackage[utf8]{inputenc}
\usepackage[T1]{fontenc}

\usepackage{hyperref}

\usepackage[finnish,english]{babel}

\usepackage[left=2.8cm,right=2.8cm,top=3cm,bottom=3cm]{geometry}

\usepackage{amsfonts}
\usepackage{amsmath}
\usepackage{amssymb}
\usepackage{amsthm}
\usepackage{amsxtra}

\usepackage{verbatim}

\usepackage{pgf,tikz}
\usepackage{mathrsfs}
\usepackage{float}
\usepackage{url}
\usepackage{capt-of}
\usepackage{relsize}
\usepackage{rotating}
\usepackage{graphicx}
\usepackage{mathtools}
\usepackage{epigraph}
\usepackage[square,sort,comma,numbers]{natbib}
\usepackage{makeidx}
\usepackage[xindy, toc]{glossaries}
\usepackage[toc]{appendix}
\usepackage{upgreek}
\usepackage{bbm}
\newtheorem{thm}{Theorem}[section]
\newtheorem{theorem}{Theorem}[section]

\newtheorem{proposition}[thm]{Proposition}
\newtheorem{lemma}[thm]{Lemma}
\newtheorem{corollary}[thm]{Corollary}

\newtheorem{notation}[thm]{Notation}

\newtheorem{defi/}[thm]{Definition}

\newenvironment{definition}
{\pushQED{\qed}\begin{defi/}}
	{\popQED\end{defi/}}

\theoremstyle{remark}
\newtheorem*{remark}{Remark}

\newtheorem{assumption}{Assumption}

\numberwithin{equation}{section}

\newcommand{\N}{\mathbb{N}}

\newcommand{\R}{\mathbb{R}}

\newcommand{\K}{\mathbb{K}}

\newcommand{\inner}[2]{\left\langle #1 , #2\right\rangle}

\newcommand{\wkst}{weak*}

\newcommand{\rme}{\mathrm{e}}

\renewcommand{\epsilon}{\varepsilon}
\renewcommand{\phi}{\varphi}

\newcommand{\bern}[1]{\mathscr{B}[#1]}

\newcommand{\cf}[1]{{\mathbbm 1}_{\{#1\}}}

\newcommand{\rmd}{\mathrm{d}}
\newcommand{\dv}[1]{\frac{\rmd}{\rmd #1}}
\newcommand{\abs}[1]{\left\lvert #1 \right\rvert}
\newcommand{\norm}[1]{\lVert #1 \rVert}

\DeclareMathOperator{\spt}{spt}

\DeclareMathOperator{\dist}{dist}

\newcommand{\posrad}[1]{\mathcal{M}_{+}\!\left(#1\right)}
\newcommand{\posradb}[1]{\mathcal{M}_{+,b}\!\left(#1\right)}

\newcommand{\email}[1]{E-mail: \tt #1}

\newcommand{\emailmarina}{\email{marina.ferreira@math.univ-toulouse.fr}}
\newcommand{\UTaddress}{\em CNRS $\&$ Institut de Mathématiques de Toulouse, UMR 5219, \\  \em
Université Paul Sabatier, 118 route de Narbonne,
31062 Toulouse Cedex 9, France 
 }

\newcommand{\emailaleksis}{\email{aleksis.vuoksenmaa@helsinki.fi}}
\newcommand{\UHaddress}{\em University of Helsinki, Department of Mathematics 
and Statistics\\
\em P.O. Box 68, FI-00014 Helsingin yliopisto, Finland}

\title{Smoluchowski coagulation equation with a flux of dust particles}
\author{Marina A. Ferreira \thanks{\emailmarina}, Aleksis Vuoksenmaa \thanks{\emailaleksis}
\\[1em]%
$\,^*$\UTaddress
\\[1em]%
$\,^\dag$\UHaddress}

\begin{document}
	
	\maketitle

        \begin{abstract}
    We construct a time-dependent solution to the Smoluchowski coagulation equation with
a constant flux of dust particles entering through the boundary at zero. The dust
is instantaneously converted into particles and flux solutions have linearly increasing mass. The construction is made for a general class of non-gelling coagulation kernels for which stationary solutions, so-called constant flux solutions, exist. The proof relies on several limiting procedures on a family of solutions of equations with sources supported on ever smaller sizes. In particular, uniform estimates on the fluxes of these solutions are derived in order to control the singularity produced by the flux at zero. We further show that, up to the multiplication by a scalar, flux solutions averaged in size and integrated in time, are bounded from above by the explicit solution, $x^{-\frac{\gamma+3}{2}}$, of the constant flux equation, where $\gamma$ is the homogeneity of the bounds of the coagulation rate kernel. Flux solutions are expected to converge to a constant flux solution in the large time limit. We show that this is indeed true in the particular case of the constant kernel with zero initial data.    
    \end{abstract}
    
\bigskip


\bigskip
 

\newpage
\tableofcontents

	\section{Introduction}
	
	The Smoluchowski coagulation equation was proposed in 1916 as a model for the evolution of the size distribution of particles as they grow upon binary coagulation \cite{smoluchowski1927three}. The number density $f_t(x)$ of particles  of size $x > 0$ at time $t\geq 0$ evolves according to the following equation
\begin{equation}
\label{eq:coagulation-equation}
\partial_t f_t = \frac{1}{2}\int_0^x K(x-y,y)f_t(x-y)f_t(y) \rmd y - \int_0^\infty K(x,y) f_t(x)f_t(y) \rmd y, 
\end{equation}
where the coagulation rate kernel $K$ is assumed to depend only on the sizes of the two coagulating particles. 
Formally, one can check that the total mass in the system described by the first moment in size, $M_1(f_t) := \int_0^\infty xf_t(x)\rmd x$, remains constant over time, and many previous works have studied the properties of such solutions,  so-called {\it mass-conserving solutions} \cite{banasiak2019analytic}. 
On the other hand, solutions with increasing mass due to an incoming flux of mass from zero, that we
call \textit{flux solutions}, can also exist, and they are the focus of the present article.

%See the recent works \cite{ferreira_stationary_2021, ferreira2023non} where existence and non-existence of stationary solutions are obtained for general kernels and \cite{ferreira_self-similar_2022} where a self-similar solution with zero initial data is  constructed for homogeneous kernels. 
% Escobedo and Mischler proved in 2006 \cite{e3} an existence result for the evolution problem {\rd to be checked} and very recently, stationary solutions and self-similar behaviour were studied in \cite{} and in \cite{}.

The mass flux  can be deduced from the continuity equation for the mass variable $xf_t(x)$ (for sufficiently regular $f$) ,
\begin{equation}\label{eq:cont_eq}
\partial_t (xf_t(x)) = - \partial_x J(x,t) \end{equation}
where the mass flux $J$ is defined by 
\begin{equation}\label{eq:flux}
J_f(x,t) = \int_0^x \int_{x-y}^\infty yK(z,y)f_t(z)f_t(y)\rmd z \rmd y.
\end{equation}
One can then consider different types of solutions to the coagulation equation with or without fluxes, by prescribing flux boundary conditions at zero and at infinity. The solutions we construct here will have a positive constant flux from zero, $J(x,\cdot)\to 1$ as $x\to 0$ and no flux at infinity. In particular, they will have linearly increasing mass.
%\begin{equation}\label{eq:mass_increase}
%\int_0^\infty x f_t(dx) =c t+\int_0^\infty x f_0(dx).
%\end{equation}

In open systems, the input of tiny particles, called {\it dust}, into the system can be modeled by such a flux boundary condition at zero.
These systems appear in physics, for example in the formation of soot from smoke produced by combustion  or in the growth of aerosols in power plant plumes due to gas-to-particle conversion \cite{fried}. 
Nevertheless, flux solutions have been almost absent from the mathematical literature (see \cite{ferreira_self-similar_2022} where a flux solution with zero initial data has been constructed for a class of homogeneous coagulation kernels).

In contrast, solutions with fluxes at infinity have been widely studied in the mathematical and applied literature. Such solutions are associated with the failure of the conservation of mass in finite time, which has been shown to occur for the so-called gelling kernels. Gelling kernels are kernels which grow sufficiently fast as $x\to \infty$ so that  particles  form a positive flux of mass leaving the system at infinity, leading the total mass in the system to decrease, a phenomenon known as {\it gelation}.  This phenomenon was first observed in explicit solutions for the multiplicative kernel $K(x,y)=xy$ \cite{mcLeod1962infinite, leyvraz1981singularities} and it was later shown for general classes of  kernels in coagulation-fragmentation equations \cite{jeon_existence_1998, escobedo_gelation_2002} as well as in stochastic coalescence models  \cite{fournier_ML_2009}. 
Solutions with fluxes have also been constructed for other homogeneous kinetic equations, such as the Uehling–Uhlenbeck equation (or quantum Boltzmann equation) \cite{escobedo2012classical}. We note that solutions with fluxes have also been widely studied in the physics literature, in particular in the context of wave turbulence \cite{balk1998stability}.

%Other examples of kinetic equations with fluxes and their importance in the theory of wave turbulence are discussed in \cite{balk1998stability}.
%{\rd mention that they share some structure, maybe the singularities}

We will consider a large class of  coagulation kernels  which are bounded by below and by above by  power-laws determined by two real parameters $\gamma, \lambda$, namely, $h(x,y) : = x^{\gamma+\lambda}y^{-\lambda} + y^{\gamma+\lambda}x^{-\lambda}$.  
We will need to consider the parameter regime where solutions of the equation
\begin{equation}
J(x)=1, \quad \text{ for all }x \in (0,+\infty),
\end{equation}
so-called {\it constant flux solutions}, exist, which is exactly in the region where $|\gamma+2\lambda|<1$ \cite{ferreira_stationary_2021}. 
There is an explicit power-law solution $f(x) = c x^{-\frac{\gamma+3}{2}}$ which gives the average behaviour of any solution \cite{ferreira_stationary_2021}, which in turn yields the  asymptotic behaviour for solutions with fluxes. 
Indeed, this power-law behaviour was obtained for $x\to \infty$ for solutions featuring gelation \cite{escobedo2012classical}. In the present paper we prove that on average, solutions with flux from zero are bounded by above by the same power law. 
We notice that solutions with a smaller power law behaviour might also exist 
$f(x) = c x^{\alpha},$ $\alpha < -\frac{\gamma+3}{2}$, as $x \to 0$, but they will be mass-conserving, because they will satisfy $J(x)\to 0$ as $x\to 0$, i.e., input of mass through zero will not be strong enough to affect the mass in the system.
 On the other hand, solutions satisfying $f(x)  = c x^{\alpha},$ $\alpha > -\frac{\gamma+3}{2}$ cannot exist as the flux $J \to \infty$ as $x\to 0$. 
 Since the mass is entering through zero, one cannot in principle rule out the loss of mass at zero.  We prove that this is not the case. 
   
   The long time behaviour for  solutions with fluxes from the origin is expected to approach the self-similar solution constructed in \cite{ferreira_self-similar_2022}.
We also expect convergence, as \(t \to +\infty\), to a constant flux solution. However, this problem is not so simple for general kernels due to the lack of uniqueness of constant flux solutions \cite{ferreira_nonpower_2024}. 
In this region, constant flux solutions exist, but they are not in general unique. 
In fact, there are many kernels for which  a second solution with an oscillatory behaviour has been constructed \cite{ferreira_nonpower_2024}.
Nevertheless, in the particular case of the constant kernel $K(x,y)=2$ and zero initial data, it is possible to prove convergence to the explicit constant flux solution $f(x) = \frac{1}{2\sqrt{\pi}} x^{-\frac{3}{2}}$, see Section \ref{sec:constantKernel}.

 If  $|\gamma+2\lambda|\geq 1$ instead, then no solution is expected to exist due to the non-existence of constant flux solutions.
On the other hand, if $\gamma\geq 1$ then the mass of the solution with flux would become infinite near zero and the solution would  not be defined. %This can be expected assuming the behaviour of the time-dependent solution $f_t(x) \sim x^{-\frac{\gamma+3}{2}}$ as $x\to 0$ and the mass near zero $\int_0^1 x x^{-\frac{\gamma+3}{2}}dx < \infty$ if and only if $\gamma<1$. 
Therefore, we will assume
\begin{equation}
\gamma<1, \quad |\gamma+2\lambda|< 1.
\end{equation}
Interestingly, the equation with constant in time source $\eta$, compactly supported on $x>0$, namely
\begin{equation}\label{eq:coag_source}
\partial_t f_t = \frac{1}{2}\int_0^x K(x-y,y)f_t(x-y)f_t(y)\rmd y + \int_0^\infty K(x,y) f_t(x)f_t(y) \rmd y + \eta(x),
\end{equation}
 is proved here to have global solutions also for $\gamma>1$ and $|\gamma+2\lambda|\geq 1$, provided that the exponents in the coagulation kernel satisfy the general conditions for existence of solutions, namely $\gamma+\lambda<1$ and $-\lambda<1$. We note that a previous existence result to \eqref{eq:coagulation-equation} has also been obtained in \cite{escobedo_dust_2006} in the case of homogeneous kernels with homogeneity \(\gamma \in [0,1)\) with a time-dependent source satisfying some moment bounds.

 Whenever flux solutions exist, they are expected to describe the long-time behaviour of the equation with source for large sizes, see \cite{ferreira_self-similar_2022}.
  The long-time behaviour for \eqref{eq:coag_source}  in the case of constant kernel and continuous source term $\eta$ was considered by Dubowski{\u\i} \cite{dubovskiui1994mathematical} where convergence toward a stationary solution is proved.
 Contrarily to the regime $|\gamma+2\lambda|<1 $, in the regime $|\gamma+2\lambda|\geq 1$ the flux transporting the mass toward larger sizes is dominated by collisions between particles of very different sizes, which can yield a very rich non-trivial behaviour, including anomalous self-similar behaviour, as discussed in \cite{ferreira_coagulation_2023} using formal asymptotic expansions.

The construction of solutions with fluxes encompasses technical difficulties due to the singular behavior near zero, and typically requires the use of several limiting procedures and delicate estimates involving the flux. See, for example, the series of papers by Velázquez and Escobedo \cite{escobedo2012classical, escobedo2010fundamental, escobedo2013local} which form a whole program to construct classical local in time solutions with fluxes at infinity. 
In our proof, we first construct solutions to the equation with source supported on positive values of the size variable, 
\begin{equation}
 \partial_t f_t = \frac{1}{2}\int_0^x K(x-y,y)f_t(x-y)f_t(y)\rmd y - \int_0^\infty K(x,y) f_t(x)f_t(y) \rmd y + \frac{1}{\varepsilon} \delta_\varepsilon
\end{equation}
We then perform a careful analysis of the behaviour of the solution near zero. By relying on previous estimates for the flux obtained for constant flux solutions \cite{ferreira_stationary_2021}, it is possible to derive uniform bounds, which together with a diagonal argument and compactness tools allows to obtain a limiting time-dependent measure. This time-dependent measure is then shown to  satisfy the flux equation in some suitable sense (see definition \ref{def:weak-flux-solution} below).

As a final remark, we note that as of today a general uniqueness theory for coagulation equations is still incomplete. The uniqueness results available mainly concern mass-conserving solutions. Even in that case, they require more assumptions on the moments of the initial data and of the solutions than the ones needed to prove existence \cite{norris_smoluchowskis_1999, throm_uniqueness_2023, fournier2006well}. 
In the case of solutions with fluxes, the only uniqueness results available to our knowledge are for explicitly solvable
kernels, namely, in the case of flux at infinity for the multiplicative kernel $K(x, y) = xy$ \cite{normand2011uniqueness} and in the case of flux at zero for the constant kernel $K(x, y) = 2$ (see Section \ref{sec:constantKernel}). Some non-uniqueness results for the time-evolution   were also obtained \cite{escobedo_gelation_2003, norris_smoluchowskis_1999}.   
The construction of solutions with fluxes from zero contributes to shed light into this issue, as the presence of fluxes may be a cause of non-uniqueness \cite{ferreira_nonpower_2024}.

%, even in mass-conserving solutions. Indeed, it is in principle possible that a flux of particles entering the system is so small that it  only affects the small moments of the solution, while the mass in the system  remains constant.

%, even in mass-conserving solutions. Indeed, it is in principle possible that a flux of particles entering the system is so small that it  only affects the small moments of the solution, while the mass in the system  remains constant.

\paragraph{Acknowledgments}
We would like to thank Juan Velázquez for the main idea for the proof of Lemma \ref{lemma:power_estimate_f_eps}, as well as for suggesting the references on the construction of solutions with fluxes in coagulation equations. We would also like to thank Eugenia Franco, Petri Laarne, Jani Lukkarinen, Sakari Pirnes, and Juan Velázquez for the several discussions relevant to the project.

%The work of AV has been supported by the Academy of Finland, via an Academy project (project No. 339228) and the {\em Finnish Centre of Excellence in Randomness and Structures\/} (project No. 346306).

 The authors gratefully acknowledge the support of the Academy of Finland, via an Academy project (project No. 339228) and the Finnish centre of excellence in Randomness and STructures (project No. 346306). The research of M. A. Ferreira has also been partially funded by the Faculty of Science of University of Helsinki Support Funding 2022, ERC Advanced Grant 741487, the Centre for Mathematics of the University of Coimbra - UIDB/00324/2020 (funded by the Portuguese Government through FCT/MCTES, \url{https://doi.org/10.54499/CEECINST/00099/2021/CP2783/CT0002}) and the Starting Package 2024 funded by the Centre National de la Recherche Scientifique. The funders had no role in study design, analysis, decision to publish, or preparation of the manuscript.

\section{Main results}

    Before stating the main definitions and results of the article, we will introduce some notation that will be used throughout the proofs.

	\subsection{Notation}
	
	We use the notation \(\R_* = (0,+\infty)\) for the strictly positive real numbers. This will be the space on which the constructed measures will be defined. For the non-negative real numbers, we use \(\R_+ = [0,+\infty)\). Given a subset \(X \subset \R_*\), the set \(\posrad{X}\) denotes the set of all positive Radon measures on \(X\), while the measures in \(\posradb{X}~\subset~\posrad{X}\) are additionally bounded, which means means that any \(\mu \in \posradb{X}\) has to satisfy \(\mu(X) <+\infty\).
	
	If \(\mu \in \posrad{X}\), let \(\mathcal{N}_\mu\) be the collection of those open sets \(N\) for which \(\mu(N) = 0\). The  \emph{support} of the measure \(\mu\) is then defined as 
	\begin{align}
	\label{defeq:support-of-radon-measure}	
	\spt(\mu) \coloneqq X \setminus (\cup_{N\in \mathcal{N}_\mu} N).
	\end{align}
	This is a closed set.
	
	When \(\mu\) is a Radon measure on \(\R_*\) and \(E \subset \R_*\) is a Borel set, we use the notation \(\mu\vert_{E}\) to denote the Radon measure on \(\R_*\) given by \(\mu\vert_{E}(A) = \mu(A \cap E)\) whenever \(A\) is a Borel set. If \(E\) is closed, then any \(\mu \in \posradb{\R_*}\) with \(\spt(\mu) \subset E\) is in one-to-one correspondence with a measure \(\mu' \in \posradb{E}\). We will sometimes move freely between these two pictures.
	
	Throughout the article, \(C\) will denote a constant that is allowed to change from line to line. Sometimes we will need to highlight its dependence on some parameter, and in this case we write the parameters in the subscript -- for instance, \(C_a\) is a constant that depends on the parameter \(a\).
	
	When \(\phi\) is a function in a suitable space (such as \(C_c(X)\) or \(C_0(X)\)) and \(\mu\) is a positive Radon measure on \(X\), we use the dual pairing notation to denote the integral of the function \(\phi\) against the measure \(\mu\):
	\begin{align}
	\inner{\phi}{\mu} = \int_{X} \phi(x)\mu(\rmd x).
	\end{align} In rare occasions we also use the same notation when \(\phi \in C_b(X)\), provided that \(\mu\) satisfies strong enough integrability conditions for this to be well defined.
	
	Given \(z \in \R_*\), we will denote
	\begin{align}
	\Omega_z \coloneqq \{(x,y) \in \R_*^2 \colon 0 < x \leq z, z-x < y\}.
	\end{align}
	This subset will play an important role in what follows, as it appears in the mass flux term. For each \(z \in \R_*\), it contains those particle sizes that coagulate together in a way that makes the system lose mass in the region \((0,z]\). For suitably regular measures \(\mu\), we will denote
	\begin{align}
	J_{\mu}(z) \coloneqq \iint_{\Omega_z}xK(x,y)\mu(\rmd x)\mu(\rmd y).
	\end{align}
	
	If \(\phi \colon X \subset \R_+ \to \R\) is a function, we denote by \(D[\phi]\) the function
	\begin{align}
	D[\phi] &\colon X^2 \to \R \\
	D[\phi](x,y) &= \phi(x+y) - \phi(x) - \phi(y).
	\end{align}

	\subsection{Definition and main results}
	
	\begin{assumption}[Kernel regularity]
	\label{assumption:kernel-regularity}	
	The coagulation kernel \(K \colon \R_*^2 \to [0,+\infty)\) is assumed to be \emph{continuous}. It is also assumed that there exist constants \(0<c_1\leq c_2 <+\infty\) so that the following upper and lower bounds are satisfied for all \(x,y\in \R_*\)
	\begin{align}
	\label{ineq:kernel-bounds}
	c_1 \left(x^{\gamma+\lambda}y^{-\lambda} + x^{-\lambda}y^{\gamma+\lambda}\right) \leq K(x,y) \leq c_2\left(x^{\gamma+\lambda}y^{-\lambda} + x^{-\lambda}y^{\gamma+\lambda}\right).
	\end{align}	 
	\end{assumption}

	Note that in Assumption \ref{assumption:kernel-regularity}, the upper and lower bounds of the kernel are characterized by some power law behaviour in the variables \(x\) and \(y\). The specific choice will have a drastic effect on the behaviour of the solutions to the corresponding coagulation equation. We will impose the following restrictions.
	\begin{assumption}[Kernel exponents]
	\label{assumption:kernel-exponents}
	The exponents \(\gamma,\lambda \in \R\), describing the boundedness properties of the coagulation kernel, are assumed to satisfy all of the following conditions
	\begin{itemize}
	\item[1.] \(\abs{\gamma+2\lambda}< 1\), and
	\item[2.] \(\gamma < 1 \).
	\end{itemize}
	\end{assumption}
	
	The first condition guarantees the existence of stationary solutions to the coagulation equation with source \cite{ferreira_stationary_2021}, suggesting a non-trivial behaviour for the flux term in the long time limit. It also allows to control the contribution of the non-compact part of the flux term.  The second condition prevents gelation, in this context meaning a less-than linear growth of mass. It also guarantees that the measure isn't too singular near zero. Together, these conditions imply that \(-\lambda < 1\) and \(\gamma + \lambda < 1\), which have been considered necessary for a solution to \eqref{eq:coagulation-equation} to exist.

	We are now ready to state the definition of a weak flux solution to the coagulation equation.	
	\begin{definition}[Flux solution, weak formulation]
		Let \(K\) satisfy Assumption \ref{assumption:kernel-regularity}. We say that a time-dependent measure \(f \in C([0, T], \posrad{\R_*})\) is a weak flux solution to the coagulation equation \eqref{eq:coagulation-equation} with the initial data \(f_0 \in \posrad{\R_*}\) that satisfies \(xf_0 \in \posradb{\R_*}\), in case
		\begin{itemize}
		\item[(i)] \(xf \in C([0, T], \posradb{\R_*})\), with the notion of continuity interpreted in terms of \wkst-topology on the target space,
		\item[(ii)] for almost every \((t, z) \in [0,T] \times \R_*\)
		\begin{align}
		\int_{(0,z]} x f_t(\rmd x) - \int_{(0,z]} x f_0(\rmd x) = -\int_{0}^t J_{f_s}(z) \rmd s + t.	
		\label{evol-eq:weak-flux-solution}
		\end{align}
		
		Here the time-integrated flux term
		\begin{align}
		\int_{0}^t J_{f_s}(z) \rmd s \coloneqq \int_{0}^t \iint_{\Omega_z} xK(x,y) f_s(\rmd x)f_s(\rmd y) \rmd s
		\end{align}
		is finite for all \(t \in [0,T]\) and for almost every \(z \in \R_*\).
		\end{itemize}
		\label{def:weak-flux-solution}
	\end{definition}

	\begin{theorem}
	\label{thm:existence-of-weak-flux-solution}
	Assume that Assumptions \ref{assumption:kernel-regularity} and \ref{assumption:kernel-exponents} are satisfied. Given an initial data \(f_0 \in \posrad{\R_*}\) such that the first moment measure satisfies \(xf_0 \in \posradb{\R_*}\), there exists a weak flux solution to \eqref{eq:coagulation-equation} in the sense of Definition \ref{def:weak-flux-solution}.
	\end{theorem}
	
	\begin{theorem}
	\label{thm:linear-increase}
	If \(f\) is a solution in the sense of Definition \ref{def:weak-flux-solution}, then for almost every \(t \in [0,T]\), we have
	\begin{align}
	\label{eq:mass-increase}
	M_1(f_t) = M_1(f_0) + t.
	\end{align}

	There also exists a constant \(C\), only depending on \(T\), the kernel \(K\), exponents \(\gamma,\lambda\) and the initial data \(f_0\), such that for almost every \((t, z) \in [0,T] \times \R_*\),
	\begin{align}
	\label{eq:flux-eq-average-integral}
	\int_{0}^t \frac{1}{z} \int_{[z/2,z]} x^{-\frac{3+\gamma}{2}} f_s(\rmd x) \rmd s \leq C.
	\end{align}
	
	Moreover, if \(f\) additionally satisfies equation \eqref{evol-eq:weak-flux-solution} for all \(t \in [0,T]\), then equation \eqref{eq:mass-increase} holds for all \(t\), and if \eqref{evol-eq:weak-flux-solution} holds for all \((t, z)\), then \eqref{eq:flux-eq-average-integral} holds for all \((t, z)\). 
	\end{theorem}
	
	Together, Proposition \ref{prop:time-average-flux-bound} and \ref{prop:linear-increase-mass} constitute the proof of Theorem \ref{thm:linear-increase}.
	
	\begin{theorem}
	\label{thm:constant-kernel-convergence}
	If the coagulation kernel is constant, \(K(x,y) \equiv 2\), there exists a unique solution \(f_t\) to the flux equation -- precise formulation given later -- with the initial data \(f_0 = 0\). This solution converges weakly as a measure on \(\R_*\) to the  stationary solution of the flux equation, i.e.,
	\begin{align}
	f_t(\rmd x) \to \frac{1}{\sqrt{2\pi}} x^{-\frac{3}{2}} \rmd x, \quad t \to \infty.
	\end{align}
	\end{theorem}
	\begin{remark}
	The specific value of the constant \(c\) for which \(K \equiv c\) does not affect the result significantly, although it does change the solution \(f_t\) as well as the corresponding stationary solution. This change, however, is easy to track in the proof, so we omit it. The stationary solution is only effected by a change in the prefactor where \(\frac{1}{\sqrt{2\pi}}\) changes to \(\frac{1}{\sqrt{c \pi}}\).
	\end{remark}
	The proof of Theorem \ref{thm:constant-kernel-convergence} is postponed until Section \ref{sec:constantKernel} and follows from Propositions \ref{prop:existence-constant-kernel}, \ref{prop:stationary-solution:unique}, \ref{prop:stationary-constant-flux} and \ref{prop:constant-kernel-convergence-stationary}.

	\section{Existence of a solution to the coagulation equation with a compactly supported source}
	
	Before we construct the solution to the flux problem, we will start by finding a solution to the coagulation equation with a compactly supported source. Such a solution will be a time-dependent measure \(f \in C([0, T], \posradb{\R_*})\), where the target space is equipped with the \wkst-topology. More precisely, we make the following definition.
	\begin{definition}
	Assume that Assumption \ref{assumption:kernel-regularity} holds. We say that a time-dependent measure \(f \in C([0, T], \posradb{\R_*})\) satisfies the coagulation equation with a source term \(\eta \in \posradb{\R_*}\) with finite first moment and with \(\spt(\eta) \subset [a,M]\) and with initial data \(f_0 \in \posradb{\R_*}\) with a finite first moment with \(\spt(f_0) \subset [a,+\infty)\), in case \(\sup_{t \in [0,T]} M_1(f_t) \leq C_T\), \(\int_{0}^T M_{-\lambda}(f_t) \rmd t < + \infty\), \(\int_{0}^T M_{\gamma+\lambda}(f_t) \rmd t < + \infty\) and \(\spt(f_t) \subset [a,+\infty)\) for all \(t \in [0,T]\), and if for every test function \(\phi \in C^1([0,T], C_c(\R_*))\), we have
	\begin{align}
	&\int_{\R_*}\phi(t,x)f_t(\rmd x) - \int_{\R_*}\phi(0,x)f_0(\rmd x) = \int_{0}^t \rmd s \int_{\R_*} \dot{\phi}(s,x) f_s(\rmd x) \nonumber \\
	&\int_{0}^t \rmd s \int_{\R_*}\int_{\R_*} \left(\phi(s,x+y)-\phi(s,x)-\phi(s,y)\right)K(x,y)f_s(\rmd x) f_s(\rmd y) \nonumber \\
	&+ \int_{0}^{t} \rmd s \int_{\R_*}\phi(s,x)\eta(\rmd x)
	\label{eq:compactly-supported-source}
	\end{align}
	\label{def:compactly-supported-source}
	\end{definition}

	The following proposition shows that if we have a suitable initial data \(f_0\) and if the source term in the coagulation equation is compactly supported and a bounded Radon measure, then we find \emph{one} (not necessarily unique) solution to the coagulation equation inside any time interval \([0,T]\). With a simple diagonal trick, this can of course be extended to the whole half-line of positive real numbers, if the test functions are then taken to be compactly supported in time. We will such a proof from this article, as we are only interested in the existence of the solutions for some times and do not study their long time behaviour.
	
	\begin{proposition}
	\label{proposition:compactly-supported-solution}
	Assume that Assumption \ref{assumption:kernel-regularity} holds and additionally the exponents in the coagulation kernel \(K\) satisfy \(-\lambda, \gamma + \lambda < 1\). Let \(f_0 \in \posradb{\R_*}\) be the initial data, with \(\spt(f_0) \subset [a,+\infty)\) for some \(a > 0\). Assume that  \(\eta \in \posradb{\R_*}\) is a source term with \(\spt(\eta) \subset [a,+\infty)\). Then, for every \(T > 0\), there exists a solution to the coagulation equation \(f \in C([0,T], \posradb{\R_*})\) in the sense of Definition \ref{def:compactly-supported-source} with the initial data \(f_0\) and the source term \(\eta\).
	\end{proposition}

	This has the immediate consequence that there exist solutions to the coagulation equation with the initial data \(f_0^\epsilon\) and the source \(\eta_\epsilon = \frac{1}{\epsilon}\delta_\epsilon\).
	\begin{corollary}
	\label{corollary:epsilon-solution}
		Let \(f_0^\epsilon = f_0\vert_{[\epsilon,+\infty)}\), which can be identified with a measure in \(\posradb{\R_*}\). Then there exist a solution to the coagulation equation with the source \(\eta_\epsilon\) and initial data \(f_0^\epsilon\) in the sense of Definition \ref{def:compactly-supported-source}.
	\end{corollary}

	To prove Proposition \ref{proposition:compactly-supported-solution}, we first prove the existence of solutions to a truncated problem, after which we will show that we will find a solution to \eqref{eq:compactly-supported-source} as a limit point of such truncated solutions.
		
	\begin{definition}[Truncated coagulation equation]
	\label{def:compactly-supported-source-truncated}
	Let \(a>0\) and \(\mathcal{K} \geq \max(2,2a)\). Let \(\zeta_{\mathcal{K}} \in C_c(\R_*)\) be a function with \(\zeta_{\mathcal{K}} \succ \cf{[a, {\mathcal{K}}]}\) and \(\cf{[a/2,2\mathcal{K}]} \succ \zeta_\mathcal{K}\).	
	We say that a time-dependent measure \(f \in C([0, T], \posradb{\R_*})\) satisfies a truncated coagulation equation with a source \(\eta \in \posradb{\R_*}\) with finite first moment and support inside \([a,2\mathcal{K}]\), with initial data \(f_0 \in \posradb{\R_*}\) with a finite first moment and support in \([a,2\mathcal{K}]\), in case for every \(\phi \in C^1([0,T], C_c(\R_*))\), we have
	\begin{align}
	\label{eq:truncated-coagulation-equation}
	&\int_{\R_*}\phi(t,x)f_t(\rmd x) - \int_{\R_*}\phi(0,x)f_0(\rmd x) = \int_{0}^t \rmd s \int_{\R_*} \dot{\phi}(s,x) f_s(\rmd x) \nonumber \\
	&\int_{0}^t \rmd s \int_{\R_*}\int_{\R_*} \left(\zeta_{\mathcal{K}}(x+y)\phi(s,x+y)-\phi(s,x)-\phi(s,y)\right)K(x,y)f_s(\rmd x) f_s(\rmd y) \nonumber \\
	&+ \int_{0}^{t} \rmd s \int_{\R_*}\phi(s,x)\eta(\rmd x)
	\end{align}
	\end{definition}
	
	We will sometimes use the shorthand notation \(D_{\mathcal{K}}[\phi]\) for the function
	\begin{align}
	D_{\mathcal{K}}[\phi] &\colon \R_*^2 \to \R \\
	D_{\mathcal{K}}[\phi](x,y) &= \zeta_{\mathcal{K}}(x+y)\phi(x+y)-\phi(x)-\phi(y)
	\end{align}
	
	The following theorem follows directly from Proposition 3.6 in \cite{ferreira_stationary_2021}, so we will omit the proof from here.
	\begin{proposition}
		Suppose that the Assumption \ref{assumption:kernel-regularity} holds and additionally the exponents in the coagulation kernel \(K\) satisfy \(-\lambda < 1\) and \(\gamma+\lambda < 1\). Let \(a>0\) and \(\mathcal{K}\geq \max(2,2a)\). Then, given an initial data \(f_0 \in \posradb{\R_*}\) with \(\spt(f_0) \subset [a,2\mathcal{K}]\) and a source term \(\eta \in \posradb{\R_*}\) with \(\spt(\eta) \subset [a,2\mathcal{K}]\). Then there exists a unique time-dependent solution \(f \in C^1([0,T], \posradb{\R_*})\) to the truncated coagulation equation in a classical sense. Moreover, this solution solves \eqref{eq:truncated-coagulation-equation} in the sense of Definition \ref{def:compactly-supported-source-truncated}.
		
		Additionally, the support of the solution stays inside the following compact set:
		\begin{align}
		\label{inclusion:support_ft_truncated}
		\spt(f_t) \subset [a,2\mathcal{K}], \quad t \in [0,T].
		\end{align}
	\end{proposition}
	
	%\begin{lemma}
	%\label{lemma:solution-supported-away-from-zero}
	%A solution \(f_t\) to the truncated coagulation equation in the sense of Definition \ref{def:compactly-supported-source-truncated} is supported inside \([a,+\infty)\), where
	%\begin{align}
	%a = \min(\spt(f_0) \cup \spt(\eta)).
	%\end{align}
	%\end{lemma}
	%\begin{proof}
	%	This follows from the fixed point formulation
	%\end{proof}		
	
	\begin{lemma}
	\label{lemma:truncated-moment-bounds}	
	Any solution to \eqref{eq:truncated-coagulation-equation} in the sense of Definition \eqref{def:compactly-supported-source-truncated} satisfies the following bounds:
	\begin{align}
	\sup_{t \in [0,T] }M_{\gamma+\lambda}(f_t) \leq C <+\infty
	\end{align}
	\begin{align}
	\sup_{t \in [0,T] }M_{-\lambda}(f_t) \leq C <+\infty,
	\end{align}
	where the constant \(C\) may depend on \(f_0\), \(T\) and \(\eta\), as well as the exponents and constants of the coagulation kernel, but it does not depend on \(a\) nor \(\mathcal{K}\).
	
	\end{lemma}
	\begin{proof}
	Let \(\psi(x)=x^p\phi(x)\), where the exponent satisfies \(p \in [0,1]\) and the function \(\phi \colon \R_* \to \R\) is defined piecewise as follows:
	\begin{align}
	\phi(x) = \begin{cases}
	0, \quad &0 < x < \frac{a}{2}\\ 
	\frac{2}{a}x, \quad &\frac{a}{2} \leq x < a\\
	1, \quad &a \leq x < 2\mathcal{K}\\
	((2\mathcal{K}+1)-x), \quad &2\mathcal{K} \leq x < 2\mathcal{K}+1\\
	0, \quad  &2\mathcal{K}+1 \leq x < +\infty.
	\end{cases}
	\end{align} Since \(\phi \in C_c(\R_*)\), then also \(\psi \in C_c(\R_*)\). It is therefore a valid test function. Applying the weak formulation \eqref{eq:truncated-coagulation-equation}, we obtain
	\begin{align}
	\int_{\R_*} x^p f_t(\rmd x) &= \int_{\R_*} x^p\phi(x)f_t(\rmd x)\ \nonumber \\
	&\leq \int_{\R_*}x^p f_0(\rmd x) + T\int_{[a,+\infty)}x^p\eta(\rmd x).
	\end{align}
	From this, we obtain the following bound for the zeroeth moment of the measures \(f_t\)
	\begin{align}
	\label{ineq:zeroeth-moment-bound-for-truncated}
	&\sup_{t\in [0,T]} M_0(f_t) \leq M_0(f_0) + TM_{0}(\eta) 
	\end{align}
	as well as the \(\gamma+\lambda\) and \(-\lambda\) moments
	\begin{align}
	&\sup_{t\in [0,T]} M_{\gamma+\lambda}(f_t) \leq M_{\gamma+\lambda}(f_0) + TM_{\gamma+\lambda}(\eta) \\
	&\sup_{t\in [0,T]} M_{-\lambda}(f_t) \leq M_{-\lambda}(f_0) + TM_{-\lambda}(\eta).
	\end{align}
	\end{proof}
	\begin{remark}
	We actually have a tighter bound for all \(p \in [0,1]\), given by solving the upper bound for the differential inequality
	\begin{align}
	\dv{t}\int_{[a,+\infty)} x^p f_t(\rmd x) \leq (2^{p}-2)a^p c_1 M_{\gamma+\lambda}(f_t)M_{-\lambda}(f_t) + M_p(\eta).
	\end{align}
	\end{remark}

	\begin{lemma}
	\label{lemma:coagulation-part-estimate}
	Let \(f\) be a solution to the coagulation equation with a compactly supported source according to Definition \ref{def:compactly-supported-source-truncated}. Then it will satisfy the following bound for  all \(t \in [0,T]\)
	\begin{align}
		\int_{\R_*}\int_{\R_*}K(x,y)f_t(\rmd x)f_t(\rmd y) \leq 2 c_2 M_{\gamma+\lambda}(f_t)M_{-\lambda}(f_t).
		\label{ineq:coagulation-kernel-bound-time-slice}
	\end{align}
	\end{lemma}
	\begin{proof}
	The claim follows immediately from the upper bound in \eqref{ineq:kernel-bounds} together with Lemma \ref{lemma:truncated-moment-bounds}.
	\end{proof}
	
	We are now ready to prove Proposition \ref{proposition:compactly-supported-solution}.
	
	\begin{proof}[Proof of  Proposition \ref{proposition:compactly-supported-solution}.]
	Proposition 3.6 in \cite{ferreira_stationary_2021} shows that there exists a unique solution \(f^{\eta, \mathcal{K}} \in C^1([0, T], \posradb{\R_*})\) that satisfies Definition \ref{def:compactly-supported-source-truncated}. Here the target space \(\posradb{\R_*}\) is equipped with the total variation norm.
	
	Letting the truncation parameter \(\mathcal{K}\) vary produces a family of time-dependent measures, \(\mathfrak{F}_{\eta} \coloneqq \{f^{\eta,\mathcal{K}}\}_{\mathcal{K}\geq 2}\). We then consider the following subset of test functions. Let \(P_\eta\) be given by 
	\begin{align}
	P_\eta \coloneqq \{\phi \in C_c(\R_*) \colon \norm{\phi}_\infty \leq \frac{1}{4(C_{\eta,T} + T\tilde{C})}\},
	\end{align}
	where \(\tilde{C}=2c_2 \sup_{t\in [0,T]}M_{\gamma+\lambda}(f_t)M_{-\lambda}(f_t)<+\infty\).  On the other hand, the constant \(C_{\eta,T}\) is obtained from the uniform--in--\(\mathcal{K}\) bound for the zeroeth moments of \(f^{\eta,\mathcal{K}}\). This is possible, since \(f^{\eta,\mathcal{K}}\) is supported away from the origin.
 	
 	We can consider the set of those positive, bounded Radon measures that are supported inside \([a,+\infty)\) denote this by \[\mathscr{X}_a \coloneqq \{\mu \in \posradb{\R_*} \colon \spt(\mu) \subset [a,+\infty)\}.\] The set \(\mathscr{X}_a\) is a closed subset of \(\posradb{\R_*}\) when the latter is equipped with the \wkst-topology.
 	
 	We also define the set
	\begin{align}
	\mathscr{Y}_{T,\eta} \coloneqq \{\mu \in \posradb{\R_*} \colon \abs{\inner{\phi}{\mu}} \leq 1 \text{ for all } \phi \in P_\eta\}.
	\end{align} 
	
	The collection \(\mathfrak{F}_\eta\) satisfies the inclusion \(\mathfrak{F}_\eta \subset C([0, T], \mathscr{Y}_{T,\eta})\). This follows from Lemmas \ref{lemma:truncated-moment-bounds} and \ref{lemma:coagulation-part-estimate}.
 	
	We want to show that for any \(f_0 \in \mathscr{Y}_{T,\eta} \cap \mathscr{X}_a\), the solution to \eqref{eq:truncated-coagulation-equation} with the initial data \(f_0 \vert_{[a,2\mathcal{K}]}\) stays inside \(\mathscr{Y}_{T,\eta} \cap \mathscr{X}_a\) for all times \(t\in [0,T]\). To this end, let \(\mathcal{K} \geq 1\). By \eqref{inclusion:support_ft_truncated}, we have \(f_t \in \mathscr{X}_a\).  Pick \(\phi \in P_\eta\). Now by positivity of the measure \(f_t^{\eta,\mathcal{K}}\) and the estimate \eqref{ineq:zeroeth-moment-bound-for-truncated}, we obtain 
	\begin{align}
	\abs{\inner{\phi}{f_t^{\eta,\mathcal{K}}}} &\leq \norm{\phi}_\infty \int_{\R_*}\abs{f^{\eta,\mathcal{K}}_t}(\rmd x) \\
	&= \norm{\phi}_\infty \int_{\R_*}f^{\eta,\mathcal{K}}_t(\rmd x) \leq 1.
\end{align}	
So for all times \(t \in [0,T]\), we have \(f_t^{\eta,\mathcal{K}} \in \mathscr{Y}_{T,\eta}\). Since \(f^{\eta,\mathcal{K}}\) is additionally continuous to \(\posradb{\R_*}\) when the target space is equipped with the \wkst-topology, and especially to \(\mathscr{Y}_{T,\eta}\), it follows that \(f^{\eta,\mathcal{K}} \in C([0,T], \mathscr{Y}_{T,\eta})\).

	Equipped with the \wkst-topology, the set \(\mathscr{Y}_{T,\eta}\) becomes a complete metric space. This property is then extended to the space \(C([0,T], \mathscr{Y}_{T,\eta})\) when equipped with the uniform metric. Moreover, since the collection of time-dependent measures \(\overline{\mathfrak{F}_\eta} \subset C([0,T], \mathscr{Y}_{T,\eta})\) is equicontinuous and the time-slices 
	\begin{align}
	\overline{\mathfrak{F}_\eta}[t] \coloneqq \{f(t) \colon f \in \overline{\mathfrak{F}_\eta}\}
	\end{align}
	have compact closure, it follows by the Arzela--Ascoli theorem for metrizable spaces that the closure \(\overline{\mathfrak{F}_\eta}\) is sequentially compact. 
	
	Since for every \(C > 0\), the set
	\begin{align}
	\{\mu \in \mathscr{X}_a \colon \int_{\R_*} x \mu(\rmd x) \leq C\} \subset \posradb{\R_*}
	\end{align}
	is closed in the \wkst-topology, it follows that any \(f \in \overline{\mathfrak{F}_\eta}\) satisfies the following bound for the first moment
	\begin{align}
	\label{ineq:truncation-limit-satisfies-moment-bound}
	\int_{\R_*} x f_t(\rmd x) \leq \int_{\R_*} x f_0(\rmd x) + t \int_{\R_*} x\eta(\rmd x), \quad t \in [0,T].
	\end{align}
	Thus, for later use, we also have the estimate
	\begin{align}
	\sup_{t \in [0,T]} \int_{\R_*} x f_t(\rmd x) \leq C_T.
	\end{align}

	By sequential compactness, we can find sequence of natural numbers \((\mathcal{K}_n)_{n\in \N}\) with \(\mathcal{K}_n \to \infty\) and a limit point 
	\begin{align}
	f^\eta \in \overline{\mathfrak{F}_\eta} \subset C([0,T], \mathscr{Y}_{T,\eta}),
	\end{align} such that the following convergence holds
	\begin{align}
	f^{\eta, \mathcal{K}_n} \to f^{\eta}, \quad \text{in } C([0,T], \mathscr{Y}_{T, \eta}).
	\end{align}

	The next step is verifying that the limit measure satisfies the coagulation equation in the sense of Definition \ref{def:compactly-supported-source}. 
    %Let the subsequence of parameters \(\mathcal{K}\) be denoted by \((\mathcal{K}_n)_{n=1}^\infty\).
    Fix \(\phi \in C^1([0,T], C_c(\R_*))\). Then \(\phi = x \psi\), where \(\psi \in C^1([0,T], C_c(\R_*))\). Then let
	\begin{align}
	c_s \coloneqq \max\spt(\psi(s,\cdot)) \\
	c \coloneqq \max_{s\in [0,T]} c_s <+\infty.
	\end{align}
	
	For any \(L \geq c\), we have
	\begin{align}
	\label{eq:in-out}
	&\frac{1}{2}\iint_{\R_*^2\setminus (0,L]^2} D_{\mathcal{K}_n}[\phi_s](x,y)K(x,y)f_s^{\eta,{\mathcal{K}_n}}(\rmd x)f_s^{\eta,{\mathcal{K}_n}}(\rmd y) \nonumber \\
	&= -\iint_{(0,L] \times [L,+\infty)}x\psi(s,x)K(x,y)f_s^{\eta,\mathcal{K}_{n}}(\rmd x)f_s^{\eta,\mathcal{K}_{n}}(\rmd y).
	\end{align}
	
	Since \(-\lambda, \gamma+\lambda < 1\), then in the region \(y \in [L,+\infty)\) we can estimate \(y^{-\lambda}\) and \(y^{\gamma+\lambda}\) from above by \(L^{-\lambda-1}y\) and \(L^{\gamma+\lambda -1}y\), respectively. Therefore, the modulus of the integral on the right hand side of \eqref{eq:in-out} can be bounded from above, uniformly in \(\mathcal{K}_n\), by
	\begin{align}
	\label{expression:upper-bound-for-coagulation-tails}
	\sup_{x \in (0,c]}\abs{x^{\gamma+\lambda}\psi(s,x)} C(T)^2 L^{-\lambda-1}+ \sup_{x \in (0,c]}\abs{x^{-\lambda}\psi(s,x)} C(T)^2 L^{\gamma+\lambda-1}.
	\end{align}
	This can be made small by picking \(L\) large enough. Here, the constant \(C(T)\) has the form
	\begin{align}
	C(T) = c_2(M_1(f_0) + T),
	\end{align}	so it depends on the initial data \(f_0\), the final time \(T\) and the properties of the coagulation operator through the upper bound in \eqref{ineq:kernel-bounds}. Since \(L \geq c\) and \(x \mapsto \psi(s,x)\) is supported in \((0,c]\), we can take the supremum over this set instead of the larger set \((0,L]\). Moreover, since \(\min_{s\in [0,T]} \min(\spt(\psi(s,\cdot))) \geq b > 0\), the two suprema in \eqref{expression:upper-bound-for-coagulation-tails} are finite. 

	By \eqref{ineq:truncation-limit-satisfies-moment-bound}, similar bounds are satisfied by the limit measure \(f^{\eta}_s\).
	
	On the other hand, once \(\mathcal{K}_n \geq 2L\), the integral
	\begin{align}
	g_{\mathcal{K}_n}[\phi;f^{\eta, \mathcal{K}_n}](t) \coloneqq \frac{1}{2}\int_{0}^t \iint_{(0,L]^2} D_{\mathcal{K}_n}[\phi_s](x,y)K(x,y)f_s^{\eta,\mathcal{K}_n}(\rmd x)f_s^{\eta,\mathcal{K}_n}(\rmd y)
	\end{align}
	agrees with the following integral, where the truncated \(D_{\mathcal{K}_n}[\phi_s]\) is replaced by the limit operator \(D[\phi_s]\)
	\begin{align}
	g_{\infty}[\phi;f^{\eta,\mathcal{K}_n}](t) \coloneqq \frac{1}{2}\int_{0}^t \iint_{(0,L]^2} D[\phi_s](x,y)K(x,y)f_s^{\eta,\mathcal{K}_n}(\rmd x)f_s^{\eta,\mathcal{K}_n}(\rmd y),
	\end{align}
	and since the integrand can be approximated arbitrarily well (up to an \(\bar{\epsilon}\) of error) by a \(C_c(\R_*)\) function, it follows by the \wkst-convergence of \(f^{\eta,\mathcal{K}_{n}} \otimes f^{\eta, \mathcal{K}_n}\) towards \(f^{\eta} \otimes f^{\eta}\) together with dominated convergence  -- recall that the support of the measures has a positive distance to \(0\) -- that the following upper bound is satisfied
	\begin{align}
	&\limsup_{n\to \infty}\abs{g_{\infty}[\phi;f^{\eta,\mathcal{K}_n}](t)-g_{\infty}[\phi;f^{\eta}](t)} \leq \bar{\epsilon},
	\end{align}
	
	Since the other terms also converge
	\begin{align}
	\int_{\R_*}\phi(t,x)f_t^{\eta,\mathcal{K}_n}(\rmd x)  \to \int_{\R_*}\phi(t,x)f_t^{\eta}(\rmd x),
	\end{align}
	\begin{align}
		\int_{\R_*}\phi(0,x)f_0^{\eta,\mathcal{K}_n}(\rmd x)  \to \int_{\R_*}\phi(0,x)f_0(\rmd x),  
	\end{align}	
	and
	\begin{align}
	\int_0^{t}\int_{\R_*}\dot{\phi}(s,x)f_s^{\eta,\mathcal{K}_n}(\rmd x) \rmd s \to \int_0^{t}\int_{\R_*}\dot{\phi}(s,x)f_s^{\eta}(\rmd x) \rmd s,
	\end{align}
	it follows that the limit time-dependent measure \(f^\eta\) is a solution according to Definition \ref{def:compactly-supported-source}. In conclusion, it follows that for each \(\bar{\epsilon} > 0\), there exists \(L \geq 1\) and \(N_0 = N_0(\bar{\epsilon},L) \in \N\) such that for all \(n \geq N_0\)
	\begin{align}
	\abs{\frac{1}{2}\iint_{\R_*} D_{\mathcal{K}_n}[\phi_s](x,y)K(x,y)f^{\eta,\mathcal{K}_n}_s(\rmd x)f_s(\rmd y)-\frac{1}{2}\iint_{\R_*} D[\phi_s](x,y)K(x,y)f^{\eta}_s(\rmd x)f_s(\rmd y)} \leq 2\bar{\epsilon}
	\end{align}
	and
	\begin{align}
	\abs{\int_{\R_*}\phi(t,x)f_t^{\eta,\mathcal{K}_n}(\rmd x)-\int_{\R_*}\phi(t,x)f_t^\eta(\rmd x)} \leq \bar{\epsilon}\nonumber \\
	\abs{\int_{\R_*} \phi(0,x)f_0^{\eta,\mathcal{K}_n}(\rmd x)- \int_{\R_*}\phi(0,x)f_0^{\eta}(\rmd x)} \leq \bar{\epsilon} \nonumber \\
	\abs{\int_0^{t}\int_{\R_*}\dot{\phi}(s,x)f_s^{\eta,\mathcal{K}_n}(\rmd x) \rmd s - \int_0^{t}\int_{\R_*}\dot{\phi}(s,x)f_s^{\eta}(\rmd x) \rmd s} \leq \bar{\epsilon}.
	\end{align}
	Since \(\bar{\epsilon} \geq 0\) was arbitrary and since each \(f^{\eta,\mathcal{K}_n}\) satisfies \eqref{eq:truncated-coagulation-equation}, it follows that  the limit point \(f^{\eta}\) satisfies \eqref{eq:compactly-supported-source}.
	\end{proof}

	\section{Solution to the flux equation}
	This section is dedicated to finding a time-dependent measure that solves the flux equation in the sense of Definition \ref{def:weak-flux-solution}. We begin by outlining the main steps of the proof.
	
	\subsection{Strategy of the proof}
	Proposition \ref{proposition:compactly-supported-solution} guarantees that for any compactly supported \(\eta \in \posradb{\R_*}\), we have a time-dependent measure \(f^\eta \in C([0,T], \mathscr{Y}_{T,\eta})\) that solves the coagulation equation with a compactly supported source in the sense of Definition \ref{def:compactly-supported-source}. We can thus consider a family of such solutions, each obtained from the coagulation equation with a different source term, but all sharing compatible initial data. These source terms are taken as \(\frac{1}{\epsilon}\delta_{\epsilon}\), which means that they are supported at the point \(\epsilon \in \R_*\) and their mass is normalized. Our aim is then to find a time-dependent measure as a limit point of such a family of solutions. In order to find a bounded Radon measure in the limit, we have to regularize the solutions by multiplying them by the mass of the particles. The resulting mass measure then satisfies the weak formulation \eqref{evol-eq:weak-flux-solution} of the flux equation.
	
	Roughly, the proof can be divided into the following several steps:
	\begin{itemize}
	\item[1.] The initial data is a measure \(f_0 \in \posradb{\R_*}\), which additionally satifies the moment condition \(xf_0 \in \posradb{\R_*}\).
	\item[2.] For each \(\epsilon \in (0,1)\), consider the coagulation equation with the compactly supported source \(\frac{1}{\epsilon}\delta_\epsilon\) and the initial data \(f_0\vert_{[\epsilon,+\infty)}\). Pick one solution to the initial value problem in the sense of Definition \ref{def:compactly-supported-source}. This is possible, since Proposition \ref{proposition:compactly-supported-solution} guarantees the existence of at least one solution.
	\item[3.] For each such solution \(f^{\epsilon}\), consider the corresponding time-dependent mass measure, which we denote by \(xf^{\epsilon} \in C([0,T], \posradb{\R_*})\). The collection of these solutions \(\{x f^{\epsilon}\}_{\epsilon \in (0,1)}\) is a subset of \(\posradb{\R_*}\).
	\item[4.] For each \(M \in \N_1\), consider the family of the solutions restricted to the closed interval \(I_M = [2^{-M},2^{M}]\). Denote this by \(\mathfrak{F}_M = \{xf^\epsilon\vert_{I_M}\}_{\epsilon \in (0,1)}\).
	\item[5.] Let \(M=1\), by compactness we find a limit point  \(F^{M} \in C([0,T], \posradb{\R_*})\) and a sequence \((\epsilon^{1}_i)_{i=1}^\infty\) such that \(xf^{\epsilon^1_i}\vert_{I_1}\to F^1\). Iterate then the following procedure. On the previous step, we have found a limit point \(F^M\) and a sequence \((\epsilon^M)_{i=1}^\infty\) such that \(xf^{\epsilon^M_i}\to F^M\). Consider then the subfamily \(\{xf^{\epsilon^M_i}\vert_{I_{M+1}}\}_{i\in \N}\). By compactness, we find a limit point \(F^{M+1}\) and a subsequence \((\epsilon^{M+1}_i)_{i=1}^\infty\) such that \(xf^{\epsilon_i^{M}}\vert_{I_{M+1}} \to F^{M+1}\). Moreover, \(F^{M+1}\vert_{I_M} = F^{M}\).
	\item[6.] Take the diagonal subsequence \((\epsilon(i))_{i=1}^\infty = (\epsilon^i_i)_{i=1}^\infty\). Define a candidate limit measure  first pointwise in time by setting
	\begin{align}
		\label{eq:candidate-measure-limit-assignment}
		\inner{\phi}{F_t} = \lim_{i \to \infty} \inner{\phi}{xf^{\epsilon(i)}}, \quad \phi \in C_c(\R_*)
	\end{align}	 
	If \(M\) is such that \(\spt(\phi) \in [2^{-M},2^{M}]\), then for all \(i\) satisfying \(\epsilon(i) < 2^{-M}\), we have
	\begin{align}
	\inner{\phi}{xf^{\epsilon(i)}} = \inner{\phi}{xf^{\epsilon(i)}\vert_{I_M}},
	\end{align}
	and these we know to converge to a limit as \(i \to \infty\). The assignment \eqref{eq:candidate-measure-limit-assignment} then defines a positive, linear, bounded functional on \(C_c(\R_*)\), and therefore corresponds to a Radon measure \(F_t \in C_c(\R_*)\). As a mapping from \([0,T]\) to \(\posradb{\R_*}\), \(F\) is continuous, so we have a canditate solution \(F \in C([0,T],\posradb{\R_*})\), which we can factor into \(xf \in C([0,T],\posradb{\R_*})\).
	\item[7.] Verify that the solutions \(\{xf^\epsilon\}\) to the coagulation equation with a source in the sense of Definition \ref{def:compactly-supported-source} satisfy certain uniform boundedness properties for the time-averaged weighted dyadic moments and that these estimate transfer to any limit point.  
	\item[8.] The last step is to verify that \(f \in C([0,T],\posrad{\R_*})\) verifies the flux equation in the sense of Definition \ref{def:weak-flux-solution}. 
	\end{itemize}
	
	Steps 3--6 are the construction steps, and in the next subsection we will elaborate on the details. In the subsection after that, we obtain the estimates needed for Step 7, and finally in the last subsection we finish the proof of Theorem \ref{thm:existence-of-weak-flux-solution} by completing step 8.

	\subsection{Finding the candidate solution as a limit point}
		
	For each \(M \in \N_1\), define \(I_M \coloneqq [2^{-M},2^M]\). These form an increasing sequence of sets that satisfy
	\(\R_* = \cup_{M \geq 1} I_M\).	
	
	We find the time-dependent measure in \(C([0,T], \posradb{\R_*})\) by doing the diagonal trick for time-dependent measures in \(C([0,T], \posradb{I_M})\), i.e. by finding a sequence of measures \((xf^{\epsilon(n)})_{n=1}^\infty\) and a limit point \(F \in \posradb{\R_*}\) such that for every \(M \in \N\),  we have
	\begin{align}
	\sup_{s\in [0,T]}d(xf^{\epsilon(n)}(s), F(s)) \to 0.
	\end{align}
	Here the metric \(d = d_M\) is one that metrizes the topology on the relevant compact subset of \(\posradb{\R_*}\).
	
	Since for every \(\phi \in C_c(\R_*)\), there exists \(M \in \N_1\) such that \(\phi \in C_c(I_M)\), it follows that for all \(\phi \in C_c(\R_*)\).
	\begin{align}
		\sup_{s \in [0,T]} \abs{\inner{\phi}{xf^{\epsilon(n)}(s)\vert_{I_M}-F(s)\vert_{I_M}}} \to 0.
	\end{align}	
	
	\begin{notation}
	Here and later on we will use the following notation. If \(g \in C([0,T], \posradb{\R_*})\) and \(K \subset \R_*\), we denote by \(g\vert_{K}\) the time-dependent measure \(g\vert_{K} \in C([0,T], \posradb{\R_*})\), given pointwise in time by \(g\vert_{K}(t) = g(t)\vert_{K}\).
	\end{notation}

	\begin{definition}
	We define \(\mathscr{X}_{T,M}\) as
	\begin{align}
	\mathscr{X}_{T,M} \coloneqq \{\mu \in \posradb{I_M} \colon \abs{\inner{\phi}{\mu}} \leq 1 \text{ for all } \phi \in P_{T,M}\},
	\end{align}
	where \(P_{T,M} \coloneqq \{\phi \in C_c(I_M) \colon  \norm{\phi}_\infty \leq \frac{1}{2C_T}\}\).
	
	We then define \(\mathfrak{F}_M \coloneqq \{xf_\epsilon \vert_{I_M}\}_{0<\epsilon<1} \subset C([0,T], \mathscr{X}_{T,M})\)
	\end{definition}

	It follows from the Banach--Alaoglu theorem that the set \(\mathscr{X}\) is \wkst-compact. Moreover, the following topological result follows from the properties of the \(C_c(\R_*)\) \citep[See Thm 3.6]{rudin_functional_1974}:
	\begin{lemma}
	The \wkst-topology on the set \(\mathscr{X}_{T,M} \subset \posradb{I_M} \subset \posradb{\R_*}\) is metrizable.
	\end{lemma}
	\begin{proof}
		The space \(I_M\) is second countable, (locally) compact and Hausdorff. Therefore the space normed space \(C_c(I_M)\), equipped with the topology induced by the supremum norm, is separable. The proof of the claim now follows along the lines of the proof of Theorem 3.6 in \cite{rudin_functional_1974}. We include it here for completeness, as it also gives recipe for constructing a metric that verifies the metrizability.
		
		We can pick a dense set \(\{\phi_n\}_{n\in \N}\) in \(C_c(I_M)\). Define 
		\begin{align}
		&f_n \colon (C_c(I_M))^* \to \R \nonumber \\
		&f_n(\mu) = \inner{\phi_n}{\mu}.
		\end{align} These are all continuous with respect to the \wkst-topology on the domain. The set \(\{f_n\}_{n\in \N}\) is a countable family of continuous functions that separates points on \((C_c(I_M))^*\).
		
		Define a metric on \(\mathscr{X}_{T,M}\) by setting
		\begin{align}
		\label{defeq:metric-on-X}
		d(\mu, \nu) = \sum_{n=1}^\infty 2^{-n}\abs{f_n(\mu)-f_n(\nu)}.
		\end{align} 
		This is a metric, since the collection \(\{f_n\}_{n\in \N}\) separates points on \(\posradb{I_M}\). Since every \(f_n\) is \wkst-continuous and the series converges uniformly on \(\mathscr{X}_{T,M} \times \mathscr{X}_{T,M}\), it follows that 
		\[d\colon \mathscr{X}_{T,M} \times \mathscr{X}_{T,M} \to [0,+\infty)\] is continuous with respect to the \wkst-topology. The other direction follows from compactness.
	\end{proof}

	\begin{proposition}
	For each \(t \in [0, T]\), the family \(\mathfrak{F}[t]\) has compact closure in \(\mathscr{X}_{T,M}\)
	\end{proposition}
	\begin{proof}
	This follows immediately, since the closure of \(\mathfrak{F}[t]\) in the \wkst-topology is a subset of \(\mathscr{X}_{T,M}\) and by the Banach--Alaoglu theorem the latter is a \wkst-compact set.
	\end{proof}
	
	\begin{proposition}
		The closure satisfies \(\overline{\mathfrak{F}} \subset C([0,T], \mathscr{X}_{T,M})\).
	\end{proposition}
	\begin{proof}
	Let \(t \in [0,T]\). Fix \(\overline{F} \in \overline{\mathfrak{F}}\), and a test-function \(\phi \in C_c(\R_*)\). There exists \(F_\phi \in \mathfrak{F}\) such that
	\begin{align}
		\abs{\inner{\phi}{\overline{F}-F_\phi}} < 1/2.
	\end{align}
	
	It follows that
	\begin{align}
	\abs{\inner{\phi}{\overline{F}}} = \abs{\inner{\phi}{F_\phi} + \inner{\phi}{\overline{F}-F_\phi}} \leq C_T\norm{\phi}_\infty + \frac{1}{2}
	\end{align}
	\end{proof}
	
	\begin{corollary}
	If \(g_n \to g\) in the space \(C([0,T], \mathscr{X}_{T,M})\), then
	for all \(\phi \in C_c(\R_*)\) we have
	\begin{align}
	\lim_{n\to +\infty}\sup_{s \in [0,T]} \abs{\inner{\phi}{g_n(s)-g(s)}} = 0.
	\end{align}
	\end{corollary}
	\begin{proof}
	The convergence \(g_n \to g\) in is equivalent to 
	\begin{align}
	\sup_{s\in [0,T]} d(g_n(s),g(s)) \to 0,
	\end{align}
	where \(d\) is defined according to \eqref{defeq:metric-on-X}. It follows that for the countable dense collection \(\mathscr{S} \subset C_c(\R_*)\) of test-functions that defines the metric \(d \colon \mathscr{X}_{T,M} \times \mathscr{X}_{T,M} \to \R_+\), we have that 
	\begin{align}
	\sup_{s\in [0,T]} \abs{\inner{\phi}{g_n(s)-g(s)}} \leq 2^{C_\phi}\sup_{s\in [0,T]}d(g_n(s),g(s))\to 0, \quad \text{ as } n \to +\infty.
	\end{align}
	On the other hand, if \(\psi \in C_c(\R_*)\), then for every \(\epsilon > 0\), we can find an \(m\) such that \(\norm{\psi-\phi_m}_\infty \leq \epsilon\), whereby
	\begin{align}
	\sup_{s\in [0,T]}\abs{\inner{\psi}{g_n(s)-g(s)}} \leq T\epsilon C + \sup_{s\in [0,T]} \abs{\inner{\phi_m}{g_n(s)-g(s)}}, 
	\end{align}
	and so
	\begin{align}
	\limsup_{n \to \infty}\sup_{s\in [0,T]} \abs{\inner{\psi}{g_n(s)-g(s)}} \leq T \epsilon C.
	\end{align}
	Since \(\epsilon\) was arbitrary, we obtain the desired result.
	\end{proof}
		
	\begin{proposition}
	There exists a time-dependent measure \(F \in C([0,T], \mathscr{X}_{T,M})\), and a monotonely decreasing sequence \((\epsilon_{n})_{n=1}^\infty\) with \(\epsilon_n \to 0\), such that
	\begin{align}
	xf^{\epsilon_{n}} \to F, \quad \text{ in } \quad C([0,T], \mathscr{X}_{T,M}),
	\end{align}
	when the target space is equipped with the \wkst-topology.
	\end{proposition}

	\begin{proof}
	Note that \(\mathscr{X}_{T,M}\) is closed in the \wkst-topology. By Arzela--Ascoli, it follows that the closure of \(\mathfrak{F}_M\) is sequentially compact. Therefore, there exists a sequence \((\epsilon_n)_{n=1}^\infty\) and a limit point \(F^{M} \in C([0,T], \mathscr{X}_{T,M})\) such that
	\begin{align}
	xf_{\epsilon_n}\vert_{I_M} \to F^M, \quad \text{ in } \quad C([0,T], \mathscr{X}_{T,M}).
	\end{align}
	
	Start from \(M = 1\) and let \(s^1 = (\epsilon^1(n))_{n\in \N}\) be the sequence along which we converge to the limit. We can then consider \(M=2\) and find a subsequence \(s^2 = (\epsilon^2(n))_{n\in \N}\) such that \(xf_{\epsilon^2(n)}\vert_{I_2}\) converges to a limit, which we denote by \(F^2 \in C([0,T], \posradb{I_2}) \subset C([0,T], \posradb{\R_*})\). Moreover, \(F^2 \vert_{I_1} = F^1\).
	
	Following this procedure, we find a diagonal sequence \((\epsilon'_n)_{n=1}^\infty = (\epsilon^n(n))_{n=1}^\infty\) and a family of time-dependent measures \(\{F^M\}_{M=1}^\infty\) such that \[xf_{\epsilon'_n}\vert_{I_M} \to F^M\] for all \(M\), and \(F^M\vert_{I_k} = F^k\) for all \(k \leq M\).	
	
	Now define
	\begin{align}
	F \colon [0,T] \to \posradb{\R_*}
	\end{align}
	by setting, for each \(\phi \in C_c(\R_*)\)
	\begin{align}
	\label{defeq:definition-of-diagonal-limit-measure}
	\inner{F_t}{\phi} = \lim_{M\to \infty}\inner{F^M_t}{\phi}.
	\end{align}
	For each \(t \in [0,T]\), \(\phi \in C_c(\R_*)\), the sequence \((\inner{F^M_t}{\phi})_{M\geq 1}\) is eventually constant, so the limit exists.
	
	By the Riesz--Markov--Kakutani theorem, the assingment \eqref{defeq:definition-of-diagonal-limit-measure} defines a positive Radon measure pointwise in time. Moreover, for each \(\phi\), there exists \(M_0 \in \N_1\) such that for all \(t\in [0,T]\), \(\inner{F_t}{\phi} = \inner{F^M_t}{\phi}\) whenever \(M \geq M_0\). Therefore, picking \(M \geq M_0(\phi)\), we get 
	\begin{align}
	\inner{F_t-F_s}{\phi} = \inner{F_t}{\phi}-\inner{F_s}{\phi} =  \inner{F_t^M}{\phi}-\inner{F_s^M}{\phi} = \inner{F_t^M-F_s^M}{\phi}.
	\end{align}
	
	The term on the right hand side of the above chain of equations satisfies the continuity at the point \(t \in [0,T]\):
	\begin{align}
	\lim_{s\to t} \inner{F_t^M-F_s^M}{\phi} = 0.
	\end{align}
	This establishes the continuity in time of \(F\) from \([0,T]\) to the \wkst-topology of \(\posradb{\R_*}\).
	\end{proof}
	\begin{remark}
	The convergence result implies that for every \(\phi \in C_c(\R_*)\), we have
	\begin{align}
	\label{convergence-eq:epsilon-limit}
	\sup_{t \in [0,T]} \abs{\inner{\phi}{F_t-xf^{\epsilon_n}_t}} \to 0, \quad n \to \infty.
	\end{align}
	\end{remark}
	
	From this point on, we will denote by \(f\) the time-dependent measure \(f \in C([0,T], \posrad{\R_*})\) that satisfies for all \(t \in [0,T]\)
	\begin{align}
	\int_{\R_*} \phi(x) F_t(\rmd x) = \int_{\R_*} x\phi(x)f_t(\rmd x).
	\end{align}
	In other words, \(xf_t = F_t\).

	\subsection{Flux estimates}
	\label{subsect:flux-based-estimates}
	
	In this section, we will use the a priori estimates coming from the flux formulation to control certain averaged and time-integrated moments of a solution.
	Recall that in what follows \(\Omega_z = \{(x,y) \in \R_*^2 \colon x < z,\quad y\geq z-x\}\), and that for a measure \(\mu\) supported in \([a,+\infty)\) with \(a>0\) and with a finite first moment, the flux of this measure at point \(z \in \R_*\) is given by
	\begin{align}
	J_{\mu}(z) \coloneqq \iint_{\Omega_z}xK(x,y) \mu(\rmd x)\mu(\rmd y).
	\end{align}

	The measures \(xf_t^\epsilon\) have first moments that are bounded uniformly in \(t \in [0,T]\). Also the supports lie inside \([\epsilon,+\infty)\). Thus, the flux corresponding flux terms \(J_{f_t^\epsilon}(z)\) are finite for all \(t\) and all \(z\). However, for fixed \(t\) and \(z\), the supremum over \(\epsilon\) might be infinite.
	
	The following result shows that solutions to the coagulation equation with a compact support, whose existence is obtained via compactness arguments, satisfies a flux equation in the following stronger sense.
	\begin{proposition}
	For each \(\epsilon > 0\), the time-dependent measure \(f^\epsilon \in C([0,T], \posradb{\R_*})\) satisfies for every \(t \in [0,T]\), and every \(z \in \R_*\) the following equation.
	\begin{align}
	\label{eq:compactly-supported-flux-eq}
	\int_{(0,z]} xf_t^\epsilon(\rmd x) - \int_{(0,z]} xf_0^\epsilon(\rmd x) = -\int_{0}^t J_{f_s^\epsilon}(z) \rmd s + t\cf{z \geq \epsilon}
	\end{align}
	\end{proposition}
	\begin{proof}
	
	We will use the weak formulation \eqref{eq:compactly-supported-source}, which the time-dependent measure \(f^\epsilon\) satisfies. In order to do this, we need to approximate the function \(x \mapsto \cf{(0,z]}(x)\) by a continuous, compactly supported function. But since the time-evolved measure satisfies \(\spt(f_t^\epsilon) \subset [\epsilon,+\infty)\) for all \(t \in [0,T]\), we only need to approximate \(x \mapsto \cf{[\epsilon,z]}(x)\). Moreover, we may assume that \(\epsilon < z\), since otherwise all of the involved integrals vanish.
	
	 To this end, define \(\phi \colon \R_* \to \R\) by 
	 \begin{align}
	 \phi^{\epsilon,z}_\delta(x) = \begin{cases}
	 0, \quad &x \leq \epsilon-\delta \\
	 \delta^{-1}(x-(\epsilon-\delta)), \quad &\epsilon-\delta < x \leq \epsilon\\
	 1, \quad &\epsilon < x \leq z \\
	 1-\delta^{-1}(x-z), \quad &z < x \leq z+\delta \\
	 0, \quad &x > z+\delta. 
	 \end{cases}
	 \end{align} Now \(\phi^{\epsilon,z}_\delta \in C_c(\R_*)\), so we may apply the weak formulation \eqref{eq:compactly-supported-source} to it. This yields: 
	\begin{align}
	\int_{\R_*} \phi^{\epsilon,z}_\delta(x) f_t^\epsilon(\rmd x) - \int_{\R_*} \phi^{\epsilon,z}_\delta(x)f_0^\epsilon(\rmd x) &= \frac{1}{2}\int_{0}^t\iint_{\R_*^2} D[\phi^{\epsilon,z}_\delta](x,y)K(x,y)f_s^\epsilon(\rmd x)f_s^\epsilon(\rmd y) \nonumber \\
	&+ \int_{0}^t \int_{\R_*} \phi^{\epsilon,z}_\delta(x)\eta_\epsilon(\rmd x)
	\end{align}
	
	Since we have the pointwise estimate \(\abs{D[\phi^{\epsilon,z}_\delta](x,y)} \leq \phi^{\epsilon,z}_\delta(x+y) + \phi^{\epsilon,z}_\delta(x) + \phi^{\epsilon,z}_\delta(x)\) with this upper bound satisfying the integrability condition
		\begin{align} &\int_{0}^t \!\!\iint_{\R_*^2} \phi^{\epsilon,z}_\delta(x+y)K(x,y)f_s^\epsilon(\rmd x)f_s^\epsilon(\rmd y)
	+ 2\int_{0}^t \!\!\iint_{\R_*^2}\phi^{\epsilon,z}_\delta(x)K(x,y)f_s^\epsilon(\rmd x)f_s^\epsilon(\rmd y) 	\leq C_{T,\epsilon},
	\end{align}
	and the pointwise convergence \(D[\phi^{\epsilon,z}_\delta](x,y) \to -x\cf{(x,y) \in \Omega_z}\) holds for all \((x,y) \in [\epsilon,+\infty)^2\),
	dominated convergence theorem allows us to conclude that
	\begin{align}
	\int_{\R_*} x\cf{(0,z]} f_t^\epsilon(\rmd x) - \int_{\R_*}x\cf{(0,z]}f_0^\epsilon(\rmd x) = -\int_{0}^t \rmd s J_{f_s^\epsilon}(z)  + \cf{z\geq \epsilon}t.
	\end{align}
	
	\end{proof}
	
	\begin{lemma}
	\label{lemma:power_estimate_f_eps}
	Let \(f_t^\epsilon \in C([0,T], \posradb{\R_*})\) be a solution to \eqref{eq:compactly-supported-source} according to Definition \ref{def:compactly-supported-source}. Then there exists a constant \(C_T\) that along \(T\) and depends on the first moment on \(f_0\), but does not depend on \(R\) or \(\epsilon\), such that the following time-integrated avarage moment bound holds for all \(R,\epsilon > 0\)
	\begin{align}
	\label{ineq:epsilon-time-average-weighted-R-integral}
	\int_{0}^t \left(\frac{1}{R}\int_{[R/2,R]} R^{\frac{\gamma+3}{2}} f^\epsilon_s(\rmd x)\right)^2 \rmd s \leq C
	\end{align}
	\end{lemma}
	\begin{remark}
	Up to a loss of a multiplicative constant, we may also say that
	\begin{align}
	\sup_{R>0}\int_{0}^t \left(\frac{1}{R}\int_{[R/2,R]} x^{\frac{\gamma+3}{2}}f^\epsilon_s(\rmd x)\right)^{2} \rmd s \leq C
	\end{align}
	\end{remark}
	\begin{proof}
		
	From \eqref{eq:compactly-supported-flux-eq}, it follows that for every \(R>0\) and every \(t\), we have
	\begin{align}
	\label{eq:compactly-supported-flux-eq-R}
	\int_{(0,R]} xf^\epsilon_t(\rmd x) - \int_{(0,R]} x f_0(\rmd x) &= -\int_{0}^t \rmd s J_{f^\epsilon_s}(R) + \cf{R \geq \epsilon} t.
	\end{align}
	Since the constructed measure \(f_t^\epsilon\) is a positive measure, we can estimate the LHS of \eqref{eq:compactly-supported-flux-eq-R} from below by \(-\int_{(0,R]}xf_0(\rmd x)\), whereby
	\begin{align}
	\int_{0}^t \rmd s J_{f^\epsilon_s}(R) \leq \cf{R\geq \epsilon}t + \int_{(0,R]}x f_0(\rmd x).
	\end{align}
	
	Since
	\begin{align}
	J_{f^\epsilon_s}(R) \geq C' \int_{[R/2,R]} \int_{[R/2,R]} R^{\gamma+1} f^\epsilon_s(\rmd x)f_s(\rmd y) &= C' \left(\int_{[R/2,R]} R^{\frac{\gamma+1}{2}} f^\epsilon_s(\rmd x)\right)^2 \\
	&= C' \left(\frac{1}{R} \int_{[R/2,R]} R^{\frac{\gamma+3}{2}} f^\epsilon_s(\rmd x)\right)^2,
	\end{align}
	it follows that
	\begin{align}
	\label{eq:avg-in-terms-of-M1}
	C' \int_{0}^t \left(\frac{1}{R}\int_{[R/2, R]} R^{\frac{\gamma+3}{2}}f^\epsilon_s(\rmd x)\right)^2 \rmd s \leq (t + M_1(f_0)). 
	\end{align}
	\end{proof}
	\begin{remark}
	By Hölder's inequality, it also follows that for all \(\epsilon>0\), for all \(t \in [0,T]\) and \(R > 0\), the following estimate holds.
	\begin{align}
	\int_{0}^t \frac{1}{R} \int_{[R/2,R]}R^{\frac{\gamma+3}{2}} f_s^\epsilon(\rmd x) \rmd s \leq C_T
	\end{align}
	with the constant \(C_T\) defined as 
	\begin{align}
	\label{eqdef:CT}
	C_T \coloneqq \left(\frac{T+M_1(f_0)}{C'}\right)^{1/2}.
	\end{align}
	\end{remark}

	The average moment estimate shows that -- after integration over time -- the contribution of small values in the measures \(xf^{\epsilon}_s\) vanishes. This is the content of the next lemma.
	\begin{lemma}
	\label{lemma:f_eps_near0}
	There exists a constant \(\bar{C}_T\), depending also on the homogeneity \(\gamma\) and the initial data \(f_0\) through its first moment, such that for all \(\epsilon>0\) and all \(t \in [0,T]\) and all \(x_0 \in \R_*\), the following estimate is satisfied 
	\begin{align}
	\int_{0}^t \int_{(0,x_0]} x f^\epsilon_s(\rmd x)\rmd s \leq \bar{C}_T x_0^{\frac{1-\gamma}{2}}.
	\end{align}
	\end{lemma}
	
	\begin{proof}
	Hölder's inequality gives us
	
	\begin{align}
	\int_{0}^t \int_{(0,x_0]} xf_s^\epsilon(\rmd x)\rmd s \leq t^{1/2} \left( \int_{0}^t \left( \int_{(0,x_0]} x f_s^{\epsilon}(\rmd x)\right)^2 \rmd s\right)^{1/2}
	\end{align}
	
	On the other hand,
	\begin{align}
	\int_{(0,x_0]} x f_s^\epsilon(\rmd x) = \sum_{n=0}^\infty \int_{(2^{-(n+1)}x_0, 2^{-n}x_0]} xf_s^\epsilon(\rmd x)
	\end{align}
	and
	\begin{align}
	\int_{(2^{-(n+1)}x_0, 2^{-n}x_0]} xf_s^\epsilon(\rmd x) \leq 2^{-n}x_0 \int_{(2^{-(n+1)}x_0, 2^{-n}x_0]}f_s^\epsilon(\rmd x).
	\end{align}
	By triangle inequality, we have
	\begin{align}
	\left(\int_{0}^t \left(\int_{(0,x_0]} x f_s^\epsilon(\rmd x)\right)^2 \rmd s \right)^{1/2} & \leq \sum_{n=0}^\infty \left(\int_{0}^t \left(2^{-n}x_0\right)^2 \left(\int_{(2^{-(n+1)}x_0, 2^{-n}x_0]} f_s^\epsilon(\rmd x)\right)^2 \rmd s\right)^{1/2}
	\end{align}
	Upon rewriting, for each \(n \geq 0\), the squared dyadic integral as
	\begin{align}
	\left(\int_{(2^{-(n+1)x_0}, 2^{-n}x_0]} f_s^\epsilon(\rmd x)\right)^2 = \left( (2^{-n}x_0)^{1-(3+\gamma)/2} (2^{-n}x_0)^{-1}\int_{(2^{-(n+1)}x_0, 2^{-n}x_0]} (2^{-n}x_0)^{\frac{3+\gamma}{2}}f_s^\epsilon(\rmd x)\right)^2
	\end{align}	and using the boundedness given by \eqref{eq:avg-in-terms-of-M1}, we get
	\begin{align}
	&\sum_{n=0}^\infty \left(\int_{0}^t \left(2^{-n}x_0\right)^2 \left(\int_{(2^{-(n+1)}x_0, 2^{-n}x_0]} f_s^\epsilon(\rmd x)\right)^2 \rmd s \right)^{1/2} \nonumber \\
	&\leq \sum_{n=0}^\infty  \Big(\int_{0}^t \left(2^{-n}x_0\right)^2  (2^{-n}x_0)^{2-(3+\gamma)} \frac{t+M_1(f_0)}{C'}\Big)^{1/2} \nonumber \\
	&\leq \sum_{n=0}^\infty\left( (2^{-n}x_0)^{1-\gamma}\frac{1}{C'}\left(t+M_1(f_0)\right)\right)^{1/2} \leq C_{t,f_0} x_0^{\frac{1-\gamma}{2}}.
	\end{align}
	
	Consequently,
	\begin{align}
	\int_{0}^t \int_{(0,x_0]} x f_s^\epsilon(\rmd x) \rmd s \leq t^{1/2}\left(\int_{0}^t \left(\int_{(0,x_0]}xf_s^\epsilon(\rmd x)\right)^2 \rmd s\right)^{1/2} \rmd s \leq C_{t,f_0}x_0^{(1-\gamma)/2}
	\end{align}
	The constants \(C_{t,f_0}\) are given by
	\begin{align}
	C_{t,f_0} = t^{1/2}\sum_{n=0}^{\infty}(2^{(1-\gamma)/2})^{-n} \left(\frac{t+M_1(f_0)}{C'}\right)^{1/2} = \frac{t^{1/2}}{1-2^{(1-\gamma)/2}}\left(\frac{t+M_1(f_0)}{C'}\right)^{1/2}
	\end{align} are increasing in \(t\), so we may pick 
	\begin{align}
	\label{eq:CT_def}
	\bar{C}_T = \frac{T^{1/2}}{1-\sqrt{2}^{1-\gamma}} C_T = \frac{T^{1/2}}{1-\sqrt{2}^{1-\gamma}}\left(\frac{T+M_1(f_0)}{C'}\right)^{1/2}
	\end{align}	
	\end{proof}
	
	For the statement of the next proposition to make sense, we note that by  \(f \in C([0,T], \posrad{\R_*})\), we denote the measure which is defined pointwise in time as \(\int_{\R_*}\phi(x)F_s(\rmd x) = \int_{\R_*} \phi(x)xf_s(\rmd x)\) for all \(\phi \in C_c(\R_*)\). Essentially, this is the measure \(f_s = F_s/x\).
	
	\begin{proposition}
	There exists a constant \(C'_T > 0\) defined by \eqref{eqdef:CT} bounds the limit time-dependent measure \(f \in C([0, T], \posrad{\R_*})\) in the following sense. For every \(R > 0\), we have
	\begin{align}
	\int_{0}^t \left(\frac{1}{R}\int_{[R/2,R]} R^{\frac{\gamma+3}{2}}f_s(\rmd x)\right)^2 \rmd s \leq C'_T
	\end{align}
	holds for all \(t \in [0,T]\)
	\end{proposition}
	\begin{proof}
	From Lemma \ref{lemma:power_estimate_f_eps}, we know that for every \(R\) the upper bound
	\begin{align}
	\sup_{\epsilon >0 }\int_{0}^t \left(\frac{1}{R}\int_{[R/2,R]} R^{\frac{\gamma+3}{2}} f^{\epsilon}_s(\rmd x)\right)^2 \rmd s \leq C_T 
	\end{align}
	is satisfied.
	
	We now fix \(R > 0\) and \(0< \delta < \frac{R}{4}\). Then we pick a smooth, positive function that vanishes outside \([R/4, 2R]\) and is constant \(1\) in \([R/2,R]\). Denote this function by \(\phi_{R,\delta}\).  Now, by positivity of the measures,
	\begin{align}
	\int_{0}^t \left(\int_{[R/2,R]} R^{\frac{\gamma+1}{2}}f_s(\rmd x)\right)^2 \!\!\rmd s \leq \int_{0}^t \left(\int_{\R_*} \phi_{R,\delta}(x)R^{\frac{\gamma+1}{2}}f_s(\rmd x)\right)^2 \rmd s
	\end{align}
	
	On the other hand, by Fatou's lemma
	\begin{align}
	\int_{0}^t \left(\int_{\R_*} \phi_{R,\delta}(x)R^{\frac{\gamma+1}{2}}f_s(\rmd x)\right)^2 \rmd s &\leq \liminf_{n\to \infty} \int_{0}^t \left(\int_{\R_*} \phi_{R,\delta}(x)R^{\frac{\gamma+1}{2}}f^{\epsilon(n)}_s(\rmd x)\right)^2 \rmd s.
	\end{align}
	
	Since \(\phi_{R,\delta}\) is supported in \([R/4,2R]\), we have
	\begin{align}
	\liminf_{n\to \infty} \int_{0}^t \left(\int_{[2^{-K}, 2^K]} \phi_{R,\delta}(x)R^{\frac{\gamma+1}{2}}f^{\epsilon(n)}_s(\rmd x)\right)^2 \rmd s &\leq \liminf_{n\to \infty} \int_{0}^t \left(\int_{[R/4,2R]}R^{\frac{\gamma+1}{2}}f^{\epsilon(n)}_s(\rmd x)\right)^2 \rmd s.
	\end{align}
	
	The integral can be expressed as
	\begin{align}
	\int_{[R/4, 2R]} R^{\frac{\gamma+1}{2}}f_s^{\epsilon(n)}(\rmd x) = \sum_{i=0}^{2} \int_{[2^{-i}R,2^{-i+1}R]} R^{\frac{\gamma+1}{2}}f_s^{\epsilon(n)}(\rmd x) 
	\end{align}
	
	Thus, picking \(C'_T = C_{\gamma,T} = \left(2^{\gamma+1} + 1 + \left(\frac{1}{2}\right)^{\gamma+1}\right) C_T\), we obtain
	\begin{align}
	\int_{0}^t \left(\frac{1}{R}\int_{[R/2,R]} R^{\frac{\gamma+3}{2}}f_s(\rmd x)\right)^2 \rmd s \leq C'_T.
	\end{align}
	\end{proof}
	
	The time-integrated flux of a time-dependent measure can be split into three distinct integrals, depending on what kind of collisions we are considering. To this end, we introduce the following notatio for each \(\delta \in (0,1)\)
	\begin{align}
	\Sigma_1(\delta) &\coloneqq \{(x,y) \in \R_*^2 \colon y \geq \frac{1}{\delta}x\}\\
	\Sigma_2(\delta) &\coloneqq \{(x,y) \in \R_*^2 \colon \delta x < y < \frac{1}{\delta}x\} \\
	\Sigma_3(\delta) &\coloneqq \{(x,y) \in \R_*^2 \colon y \leq \delta x\}.
	\end{align}
	These sets partition the open top right quarter \(\R_*^2\).
	
	Corresponding to each \(\Sigma_i(\delta)\), we write the corresponding contribution to the flux integration domain as
	\begin{align}
	\label{eqdef:flux-term-regions}
	\Omega_z^i(\delta) = \Omega_z \cap \Sigma_i(\delta), \quad i =1,2,3.
	\end{align}
			
	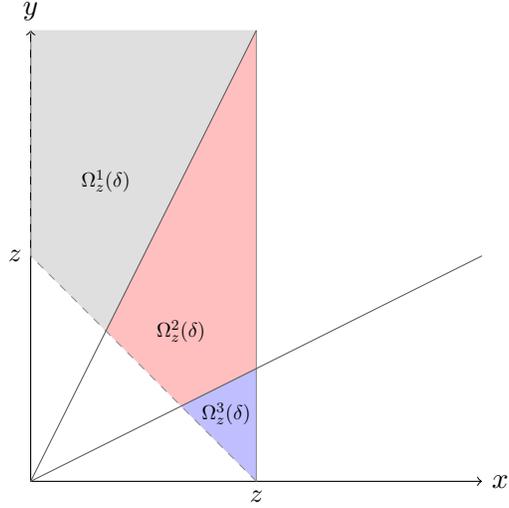
\begin{figure}[ht!]
	\label{figure:flux-term-regions}	
	\centering
	\begin{tikzpicture}
	\fill[red!25] (3,0) -- (0,3) -- (0,6) -- (3,6) |-  cycle;
	\fill[blue!25] (2, 1) -- (3,0) -- (3, 3/2);
	\fill[gray!25] (1, 2) -- (0,3) -- (0, 6) -- (3,6);	
	
		\draw[->] (0,0) -- (6,0) node[right] {$x$};
		\draw[->] (0,0) -- (0,6) node[above] {$y$};
		
		\draw[]
			node[below, scale=0.9] at (3,0) {$z$}
			node[left, scale=0.9] at (0,3) {$z$};
		
		\draw[-, color=black!45] (3,0) -- (3,6);
		\draw[dashed, color=black!45] (3,0) -- (0,3);
		\draw[dashed, color=black!45] (0,3) -- (0,6);
		
		\draw[-, color=black!65] (0,0) -- (6,3);
		\draw[-, color=black!65] (0,0) -- (3,6);
		
		\draw[]
			node[scale=0.7] at (1,4) {$\Omega_z^1(\delta)$}
			node[scale=0.7] at (2,2) {$\Omega_z^2(\delta)$}
			node[scale=0.7] at (2.6,0.9) {$\Omega_z^3(\delta)$}		
			;		
		
	\end{tikzpicture}
	\caption{Illustration of the area of integration \(\Omega_z\) appearing in the flux term, as well as the partition into \(\{\Omega_z \cap \Sigma_i(\delta)\}_{i=1,2,3}\), with \(\delta = 1/2\).}
	\end{figure}
	
	The next proposition tells us that the contributions to the flux term coming from particles in \(\Omega_z^1(\delta)\)	and \(\Omega_z^3(\delta)\) are in some sense small in the limit \(\delta \to 0\).
	
	\begin{proposition}
	\label{prop:J1_J3}
	Suppose that a time-dependent measure \(\mu \in C([0,T], \posrad{\R_*})\) satisfies the following bound for all \(t \in [0,T]\)
	\begin{align}
	\sup_{R>0}\int_{0}^t \left(\frac{1}{R} \int_{[R/2,R]} R^{\frac{\gamma+3}{2}}\mu_s(\rmd x)\right) \rmd s \leq C_T.
	\end{align}
    Assume further that the kernel \(K\), which is used to define the flux terms \(J^i_{\mu_s}\), is continuous and that the exponents of its bounds satisfy  \(\abs{\gamma + 2\lambda}<1\).
    
	Then, there exists a constant \(\tilde{C}_T\) such that  for every \(\epsilon > 0\), there exists \(\delta_\epsilon >0\), such that for every \(\delta < \delta_\epsilon\), and all \(t \in [0,T]\), we have
	\begin{align}
	\label{ineq:flux-smallness-region-one}
	\sup_{z>0}\int_{0}^t J_{\mu_s}^1(z;\delta)\rmd s \leq \epsilon \tilde{C}_T
	\end{align}
	and
	\begin{align}
	\label{ineq:flux-smallness-region-three}
	\sup_{R > 0} \frac{1}{R}\int_{[R/2,R]} \int_{0}^t J_{\mu_s}^3(z;\delta) \rmd s \rmd z \leq \epsilon \tilde{C}_T
	\end{align}		
	\end{proposition}
	\begin{remark}
	In the context of our constructed solution \(xf_t\), it follows that the corresponding time-averaged flux is finite for almost every \(z \in \R_*\).
	\end{remark}
	\begin{proof}
	This follows the same steps as the proof of Lemma 6.1 in \cite{ferreira_stationary_2021}, with the only difference in the inclusion of the time-integral and that the upper bound constant depends on \(T\), the endpoint of the time-integral.

	\end{proof}
	
	\begin{proposition}
	The time-dependent measure \(F\) satisfies, for all times \(t \in [0,T]\)
	\begin{align}
    \label{eq:near-zero-L2-estimate}
	\int_{0}^t F_s((0,z])^2 \ \rmd s  \leq C_t(f_0)z^{(1-\gamma)}. 
	\end{align}
	\end{proposition}
	\begin{proof}
	
	We establish the claim by a limit argument concerning the sets \(A_z^\delta = [\delta,z]\). We can pick a continuous and compactly supported function \(\phi_z^\delta \in C_c(\R_*)\) that takes the value \(1\) on \(A_z^\delta\). Then, by positivity of the function and the measure, together with the triangle inequality, the following holds for all \(\epsilon >0\)
	\begin{align}
		\left(\int_{0}^t \left(F_s(A^\delta_z)\right)^2 \rmd s\right)^{1/2} &\leq \left(\int_{0}^t \left(\int_{\R_*}\phi^\delta_z(x)F_s(\rmd x)\right)^2 \rmd s\right)^{1/2} \nonumber \\
		&\leq \left(\int_{0}^t \left(\inner{\phi^\delta_z}{xf_s^\epsilon - F_s}\right)^{2} \rmd s \right)^{1/2} +\left(\int_{0}^t \inner{\phi^\delta_z}{xf_s^\epsilon}^2 \rmd s\right)^{1/2}.
	\end{align}
	
	Consequently, we have the following upper bound for the expression on the left hand side of the above chain of inequalities
	\begin{align}
	\left(\int_{0}^t \left(F_s(A^\delta_z)\right)^2 \rmd s\right)^{1/2} &\leq \limsup_{\epsilon \to 0} \left(\int_{0}^t \inner{\phi^\delta_z}{xf_s^\epsilon}^2 \rmd s\right)^{1/2} \leq C_t(f_0)z^{\frac{1-\gamma}{2}}.
	\end{align}
	
	We have \(F_t((0,z]) = \lim_{\delta \to 0} F_t(A^{\delta}_z)\). Therefore, the dominated convergence theorem gives us the upper bound
	\begin{align}
	\left(\int_{0}^t F_s((0,z])^2  \rmd s\right)^{1/2} & \leq C_{t,f_0}z^{\frac{(1-\gamma)}{2}},
	\end{align}
    from which \eqref{eq:near-zero-L2-estimate} follows with a slight modification of \(C_{t,f_0}\).
	\end{proof}
	
	\subsection{The limit point verifies the flux equation}
	
	We have constructed a time-dependent measure \(f \in C([0,T], \posrad{\R_*})\) as a limit point of the solutions to the coagulation equation with the source terms tending to zero. In the previous section, we also established useful bounds for this limit point and the sequence of time-dependent measures. In this section, we will prove that it is a solution to the coagulation equation with flux from zero, in the sense of Definition \ref{def:weak-flux-solution}. 
	
	\begin{proof}[Proof of Theorem \ref{thm:existence-of-weak-flux-solution}]
	The constructed time-dependent measure \(f \in C([0,T], \posrad{\R_*})\) satisfies condition (i) in Definition \ref{def:weak-flux-solution} already by construction. It remains to show that condition (ii) is satisfied.
				
	Our aim is to prove that for any test function \(\psi \in C_c([0,T] \times \R_*)\), the following identity is satisfied.
	\begin{align}
	\label{eq:test-function-formulation-limit-flux}
	\iint_{[0,T] \times \R_*} \psi(t,z) \left(\int_{(0,z]} x f_t(\rmd x)-\int_{(0,z]} x f_0(\rmd x) + \int_{0}^t J_{f_s}(z) \rmd s - t \right) \rmd t\rmd z = 0.
	\end{align}
	
	It is already know that the corresponding result holds for the solution to the coagulation equation with the compactly supported source \(\frac{1}{\epsilon}\delta_\epsilon\), in the sense that for the solution \(f_t^\epsilon\), we have
	\begin{align}
	\label{eq:test-function-formulation-epsilon-flux}
	\iint_{[0,T] \times \R_*} \psi(t,z) \left(\int_{(0,z]} x f^\epsilon_t(\rmd x)-\int_{(0,z]} x f_0\vert_{[\epsilon,+\infty)}(\rmd x) + \int_{0}^t J_{f^\epsilon_s}(z) \rmd s - t\cf{z\geq \epsilon} \right) \rmd t\rmd z = 0.
	\end{align}	
	
	Thus, if the function
	\begin{align}
	\label{funct-eq:flux-mapping}
	(t,z) \mapsto \int_{(0,z]}xf_t(\rmd x) - \int_{(0,z]}xf_0(\rmd x) + \int_{0}^t J_{f_s}(z) \rmd s - t
	\end{align}
	is measurable and \emph{locally integrable} on \([0,T] \times \R_*\), from equation \eqref{eq:test-function-formulation-limit-flux} it follows that the claim of \eqref{evol-eq:weak-flux-solution} holding for a.e. \((t,z) \in [0,T] \times \R_*\) is true.
	
	Let us first check the measurability of the function defined by \eqref{funct-eq:flux-mapping} with respect to the Lebesgue measure on the Borel \(\sigma\)-algebras \(\mathcal{B}([0,T])\) and \(\mathcal{B}(\R_*)\) respectively. We will do this by considering all of the individual functions appearing in \eqref{funct-eq:flux-mapping} separately. 
	
	Consider first the following function, defiend on \([0,T] \times \R_*\):
	\begin{align}
	(t,z) \mapsto \int_{(0,z]} xf_t(\rmd x).
	\end{align}
	
	We will show that this is measurable by expressing it as a limit of measurable functions. To this end, for each fixed \(\delta >0 \), define the functions \(\tilde{\phi}_\delta, \phi_\delta \colon \R_* \times \R_* \to \R_+\) as 
	
	\begin{align}
	\tilde{\phi}_\delta(x,z) = \begin{cases}
	0, \quad &x<\delta \\
	\min(1,\delta^{-1}\abs{x-\delta}, \delta^{-1}\abs{z+\delta-x}), \quad &\delta \leq x <z+\delta \\
	0, \quad & x \geq z+\delta
	\end{cases}
	\end{align}
		and
	\begin{align}
	\phi_\delta(x,z) 
	= \begin{cases}
	0,\quad &z \leq \delta \\
	\tilde{\phi}_\delta(x,z), \quad &z>\delta 
	\end{cases}
	\end{align}

	It follows from the Dominated Convergence Theorem that for all \(\delta > 0\) and all times \(t\), the function
	\begin{align}
	z \mapsto \int_{\R_*} \phi^\delta(x,z) x f_t(\rmd x)
	\end{align}
	is continuous. On the other hand, for every \(\delta > 0\) and every fixed \(z \in \R_*\), the fact that \(xf_t \in C([0,T], \posradb{\R_*})\) implies that \(t \mapsto \int_{\R_*} \phi_\delta(x,z)xf_t(\rmd x)\) is a continuous function. Together, these  imply that the following function, defined on \([0,T] \times \R_*\) by
	\begin{align}
	(t,z) \mapsto \int_{\R_*} \phi_\delta(x,z) x f_t(\rmd x)
	\end{align}
	is jointly measurable on \({[0,T]} \times {\R_*}\). For all points \((t,z) \in [0,T] \times \R_*\), the Dominated Convergence Theorem shows that
	\begin{align}
	\lim_{\delta \to 0} \int_{\R_*} \phi_\delta(x,z) x f_t(\rmd x) = \int_{\R_*} \cf{(0,z]}(x) x f_t(\rmd x) = \int_{(0,z]} x f_t(\rmd x).
	\end{align} 
	Measurability is preserved under pointwise limits, whereby it follows that the function
	\begin{align}
	(t,z) \mapsto \int_{(0,z]} x f_t(\rmd x)
	\end{align}
	is jointly measurable.

	After this, we still need to show that the function
	\begin{align}
	(t,z) \mapsto \int_{0}^t J_{f_s}(z) \rmd s
	\end{align}
	is jointly measurable. We will simultaneously show that for each \(z \in \R_*\), the function \(J_{f_s}(z)\) is measurable in the time variable \(s \in [0,T]\). The argument here is that by the Monotone Convergence Theorem, applied first to the integral over \(\R_*^2\) and then to the integral over \([0,t]\), the following convergence holds:
	\begin{align}
	\int_{0}^t J_{f_s}(z) \rmd s = \lim_{n\to \infty} \int_{0}^t J^2_{f_s}(z;1/n) \rmd s.
	\end{align}
	Therefore, the problem is reduced to showing that the right hand side of the above identity is measurable.
	
	To this end, recall that \(J_{f_s}^2(z;1/n)\) is defined as
	\begin{align}
	J_{f_s}^2(z;1/n) = \iint_{\Omega_z^2(1/n)} x K(x,y)f_s(\rmd x)f_s(\rmd y),
	\end{align}
	where \(\Omega_z^2(1/n) = \Omega_z \cap \Sigma_2(1/n)\).
	
	For each fixed triple \((z,n,r) \in \R_* \times \N_1 \times \R_*\), we define the set \(\mathcal{A}_{z,n,r} \subset \R_*^2\) as
	\begin{align}
	\mathcal{A}_{z,n,r} = \{(x,y) \in \R_*^2 \colon x+y > z, x < z+r\}\cap \Sigma_{2}(1/n).
	\end{align}
	This approximates the set \(\Omega_z^2(1/n)\) in the limit \(r \to 0\). Figure \ref{figure:A_znr} illustrates one particular instance of such a set.
	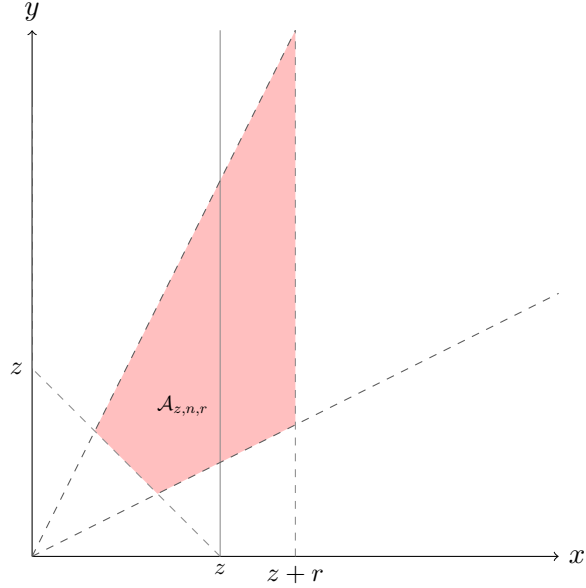
\begin{figure}[ht!]
	\label{figure:A_znr}	
	\centering
	\begin{tikzpicture}[scale=1.0]
	\fill[red!25]  (0.833,1.667) -- (1.667,0.833) -- (3.5, 1.75) -- (3.5, 7) -- (0.833,1.667) |-  cycle;
	
		\draw[->] (0,0) -- (7,0) node[right] {$x$};
		\draw[->] (0,0) -- (0,7) node[above] {$y$};
		
		\draw[]
			node[below, scale=0.8] at (2.5,0) {$z$}
			node[below, scale=0.9] at (3.5,0) {$z+r$}
			node[left, scale=0.9] at (0,2.5) {$z$};
		
		\draw[-, color=black!45] (2.5,0) -- (2.5,7);
		\draw[dashed, color=black!45] (3.5,0) -- (3.5,7);
		\draw[dashed, color=black!45] (2.5,0) -- (0,2.5);
		\draw[dashed, color=black!45] (0,2.5) -- (0,6);
		
		\draw[dashed, color=black!65] (0,0) -- (7,3.5);
		\draw[dashed, color=black!65] (0,0) -- (3.5,7);
		
		\draw[]
			node[scale=0.7] at (2,2) {$\mathcal{A}_{z,n,r}$}
			;		
		
	\end{tikzpicture}
	\caption{Illustration of the set \(\mathcal{A}_{z,n,r} \subset \R_*^2\).}
	\end{figure}
	
	Each \(\mathcal{A}_{z,n,r}\) is an open set and they also satisfy \(\bigcap_{r>0} \mathcal{A}_{z,n,r} = \Omega_z^2(1/n)\).
	We can write \(J_{f_s}^2(z;1/n)\) as the limit
	\begin{align}
	J_{f_s}^2(z;1/n) = \lim_{r\to 0} J_{f_s}^2(z;1/n,r),
	\end{align}
	where the function \(J_{f_s}^2(z;1/n,r)\) is defined as
	\begin{align}
	J_{f_s}^2(z;1/n,r) \coloneqq \iint_{\R_*^2} \cf{(x,y) \in \mathcal{A}_{z,n,r}}xK(x,y)f_s(\rmd x)f_s(\rmd y).
	\end{align}
	
	Let us now construct a sequence of functions \(\phi^{\delta,n,r} \colon \R_*^2 \times \R_* \to \R\) such that for any fized \(z \in \R_*\), the functions \(\phi^{\delta,n,r}(\cdot, z) \colon \R_*^2 \to \R\) are all continuous and compactly supported and such that for all \((x,y) \in \R_*^2\), we have \(\lim_{\delta \to 0} \phi^{\delta,n,r}((x,y),z) \to \cf{(x,y) \in \mathcal{A}_{z,n,r}}\).
	
	Indeed, we can obtain this goal by defining \(\phi^{\delta,n,r} \colon \R_*^2 \times \R_* \to \R\) pointwise as
	\begin{align}
	\phi^{\delta,n,r}((x,y),z) = \min(1,\delta^{-1}d((x,y), \R_*^2 \setminus \mathcal{A}_{z,n,r})).
	\end{align}
	
	At this stage, the parameters \(n,\delta, r\) are all fixed. It follows then from the Dominated Convergence Theorem that for each fixed \(s \in [0,T]\), the function \(\R_* \to \R\) defined pointwise by setting
	\begin{align}
	z \mapsto \iint_{\R_*^2}\phi^{\delta,n,r}((x,y), z) xK(x,y)f_s(\rmd x)f_s(\rmd y),
	\end{align}
	is continuous. On the other hand, for each fixed \(z \in \R_*\), it follows by the fact that the time-dependent measures are in \(C([0,T], \posrad{\R_*})\) that the function
	\begin{align}
	s \mapsto \iint_{\R_*^2} \phi^{\delta,n,r}((x,y),z) xK(x,y)f_s(\rmd x)f_s(\rmd y)
	\end{align}
	is continuous. Hence, the function
	\begin{align}
	(s,z) \mapsto \iint_{\R_*^2} \phi^{\delta,n,r}((x,y),z) xK(x,y)f_s(\rmd x)f_s(\rmd y)
	\end{align}
	is jointly measurable. Recall that measurability is preserved by pointwise limits. On the other hand, the dominated convergence theorem tells us that
	\begin{align}
	J_{f_s}^2(z;1/n,r) = \lim_{\delta \to 0} \iint_{\R_*^2}\phi^{\delta,n,r}((x,y),z) x K(x,y)f_s(\rmd x)f_s(\rmd y).
	\end{align}
	The resulting function is thus jointly measurable. Another application of the dominated convergence theorem gives us
	\begin{align}
	J_{f_s}^2 (z;1/n) = \lim_{r \to 0} J_{f_s}^2(z;1/n,r),
	\end{align}
	and hence the left hand side is jointly measurable. Finally, by the monotone convergence theorem, it follows that
	\begin{align}
	J_{f_s}(z) = \sup_{n\in \N_0}J_{f_s}^2(z; 1/n),
	\end{align}
	and the left hand side is also jointly measurable. Similar arguments show that the function
	\begin{align}
	(t,z) \mapsto \int_{0}^t J_{f_s}(z) \rmd s 
	\end{align}
	is jointly measurable.

	Let us now verify that \eqref{eq:test-function-formulation-limit-flux} is true for any test function \(\psi \in C_c([0,T] \times \R_*)\). The idea is to use \eqref{eq:test-function-formulation-epsilon-flux} and then control the distance between the left hand sides of these equations. To this end, let the cumulative mass of the solutions \(f_t\) and \(f_t^\epsilon\) at time \(t\) be given by
	\begin{align}
	G(t,z) &= \int_{(0,z]} xf_t(\rmd x),\\
	G^\epsilon(z,t) &= \int_{(0,z]} xf^\epsilon_t(\rmd x).
	\end{align}
	
  	For the corresponding terms at time \(t=0\), we use the notation
	\begin{align}
	C(z) &= \int_{(0,z]} x f_0(\rmd x), \\
	C^\epsilon(z) &= \int_{(0,z]} x f_0\vert_{[\epsilon,+\infty)}(\rmd x).
	\end{align}
	
	The source terms are defined by
	\begin{align}
	S(t,z) &= t, \\
	S^\epsilon(t,z) &= t\int_{(0,z]} x \frac{1}{\epsilon}\delta_\epsilon(\rmd) = t\cf{z \geq \epsilon}
	\end{align}
	
	Finally, the flux terms are defined by
	\begin{align}
	H(t,z) &= \int_{0}^t J_{f_s}(z) \rmd s \\
	H^\epsilon(t,z) &= \int_{0}^t J_{f^\epsilon_s}(z) \rmd s.
	\end{align}
		
	Having fixed this notation, let us consider a test function \(\psi \in C_c([0,T] \times \R_*)\). Our aim is to show that
	\begin{align}
	\inner{\psi}{G-C+H-S} = 0.	
	\end{align}
	We do this by using \eqref{eq:test-function-formulation-epsilon-flux}, i.e. the fact that for all \(\epsilon > 0\)
	\begin{align}
	\inner{\psi}{G^\epsilon-C^\epsilon + H^\epsilon - S^\epsilon} = 0.
	\end{align}	
	
	To connect \eqref{eq:test-function-formulation-limit-flux} and \eqref{eq:test-function-formulation-epsilon-flux}, we need to show that the terms
	\begin{align*}
	\abs{\inner{\phi}{G-G^{\epsilon(n)}}}, \abs{\inner{\phi}{C-C^{\epsilon(n)}}}, \abs{\inner{\phi}{H-H^{\epsilon(n)}}}, \abs{\inner{\phi}{S-S^{\epsilon(n)}}} 
	\end{align*}
	can be made arbitrarily small by picking a suitably large \(n\).
	
	So let \(\bar{\epsilon} > 0 \). We can find \(K \in \N\) such that
	\begin{align}
	x \mapsto \int_{0}^T\int_{(x,+\infty)} \psi(t,z) \rmd z \rmd t,
	\end{align}
	when restricted to \([2^{-K}, +\infty)\) is compactly supported and vanishes for \(x \geq 2^{K}\). On the other hand, the compactness of \(\psi\) implies that 
	\begin{align}
	\sup_{t \in [0,T]} \abs{\int_{(x,+\infty)} \psi(t,z) \rmd z} \leq \tilde{C}.
	\end{align} We therefore have the following uniform estimate
	\begin{align}
	\label{ineq:small-K}
	\sup_{\epsilon > 0}  \int_{0}^T\int_{(0,2^{-K}]} \abs{x \left( \int_{(x,+\infty)} \psi(t,z) \rmd z \right)}f_t^{\epsilon}(\rmd x)\rmd t \leq \bar{\epsilon} 
	\end{align}
	and similarly for the limit measure:
	\begin{align}
	\label{ineq:small-K-epsilon}
	\int_{0}^T\int_{(0,2^{-K}]} \abs{x \left( \int_{(x,+\infty)} \psi(t,z) \rmd z \right)}f_t(\rmd x)\rmd t \leq \bar{\epsilon} 
	\end{align}
	
	Therefore, the difference of \(\inner{\psi}{G}\) and \(\inner{\psi}{G^{\epsilon(n)}}\) can be estimated as
	\begin{align}
	\abs{\inner{\psi}{G-G^{\epsilon(n)}}} &\leq 2\bar{\epsilon}
	+\abs{\int_{0}^T \int_{[2^{-K}, 2^K]} \left(\int_{(x,+\infty)}\psi(t,z) \rmd z\right)\left(xf_t(\rmd x)-xf^{\epsilon(n)(\rmd x)}_t\right) \rmd t}
	\end{align}
	
	The convergence of \(xf^{\epsilon(n)}\) to \(xf\) in \(C([0,T], \mathscr{X}_{T,K})\) implies that for each \(\bar{\epsilon}\), there exists \(N_0\) such that
	\begin{align}
	\abs{\inner{\psi}{G-G^{\epsilon(n)}}} \leq 3\bar{\epsilon}
	\end{align}
	for all \(n \geq N_0\). Similarly, there exists a cutoff \(N_0\) such that
	\begin{align}
	\abs{\inner{\psi}{C-C^{\epsilon(n)}}} \leq 3\bar{\epsilon}
	\end{align}
	for all \(n \geq N_0\).
	
	Controlling the source is straightforward, since
	\begin{align}
	&\abs{\int_{0}^T \int_{\R_*} \psi(t,z) t \rmd z \rmd t- \int_{0}^T \int_{\R_*} \psi(t,z) t\int_{(0,z]} x \frac{1}{\epsilon}\delta_{\epsilon}\rmd z \rmd t} \nonumber \\
	&= \int_{0}^T \int_{\R_*} \abs{\psi(t,z) \left(t-t\cf{\epsilon \leq z}\right)} \rmd z \rmd t.
	\end{align}
	As \(t-t\cf{\epsilon(n) \leq z} \to 0\) as \(n \to \infty\) for all \((t,z) \in [0,T] \times \R_*\), it follows from dominated convergence that there exists \(N_0\) such that
	\begin{align}
	\abs{\inner{\psi}{S-S^\epsilon}} \leq \bar{\epsilon}, \quad n \geq N_0.
	\end{align}
	
	Finally, we can split the limit flux term \(H(t,z)\) into three different regions, corresponding to the three sets \(\Omega_z^i(\bar{\delta})\) defined in \eqref{eqdef:flux-term-regions} that partition the set \(\Omega_z\). Therefore,
	\begin{align}
	H(t,z) &= H_1(t,z) + H_2(t,z) + H_3(t,z),
	\end{align}
	with
	\begin{align}
	H_i(z,t) = \int_{0}^t J_{f_s}^i(z;\bar{\delta}) \rmd s, \quad i = 1, 2,3.
	\end{align}
	
	Similarly in the case of the compactly supported solutions, we split the flux term into
	\begin{align}
	H^\epsilon(t,z) &= H^\epsilon_1(t,z) + H^\epsilon_2(t,z) + H^\epsilon_3(t,z),
	\end{align}
	with
	\begin{align}
	H^\epsilon_i(z,t) = \int_{0}^t J_{f^\epsilon_s}^i(z; \bar{\delta}) \rmd s.
	\end{align}
	
	Since \(\psi \in C_c([0,T] \times \R_*)\) is fixed, it follows that there exist \([a,b]\) such that \(\psi(t,z) = 0\) if \(z < a\) or \(z>b\).
	
	Let \(V_x = \{(x,0) \colon x \in \R_*\}\) and \(V_y = \{(0,y) \colon y \in \R_*\}\) be the positive \(x\)- and \(y\)-axes, respectively. Recall that \(\spt(\psi(t,\cdot)) \subset [a,b]\) for all \(t \in [0,T]\). Therefore, for all all \(t \in [0,T]\) and for every point \(z \in \spt(\psi(t,\cdot))\), the distance of points in \(\Omega_z^2(\delta)\) to the \(x\)-axis is bounded from below as follows
	\begin{align}
	\dist(\Omega^2_z(\delta), V_x) \geq \frac{\delta}{1+\delta}a.
	\end{align}
	Similarly, we have the following lower bound for the distance to the \(y\)-axis
	\begin{align}
	\dist(\Omega^2_z(\delta), V_y) \geq \frac{1}{1+\frac{1}{\delta}}a = \frac{\delta}{1+\delta}a
	\end{align}
	
	Let \(\chi_1, \chi_3 \colon \R_*^2 \to \R\) be two continuous, positive functions that satisfy
	\begin{align}
	\cf{\Omega_z^1(\delta)} \leq \chi_1 \leq \cf{\Omega_z^1(\delta/2)}
	\end{align}
	and
	\begin{align}
	\cf{\Omega_z^3(\delta)} \leq \chi_3 \leq \cf{\Omega_z^3(\delta/2)}.
	\end{align}
	
	We may assume that \(\delta < 1\), and therefore \(\chi_1(x,y)\chi_3(x,y) = 0\) for all \((x,y) \in \Omega_z\). Moreover, the function
	\begin{align}
	(x,y) \mapsto \cf{(x,y) \in \Omega_z}(1-\chi_1(x,y))(1-\chi_3(x,y))xK(x,y)
	\end{align} is compactly supported in \(\R_*^2\) for every \(z \in [a,b]\). 
	
	It follows that 
	\begin{align}
	&\abs{\int_{0}^T\int_{\R_*}\psi(t,z) \int_{0}^t \iint_{\Omega_z}xK(x,y)f_s(\rmd x)f_s(\rmd y) \rmd s \rmd z \rmd t} \leq I_1 + I_2 + I_3,
	\end{align}
	with the terms on the right hand side defined as
	\begin{align}
	I_1 = \norm{\psi}_\infty \int_{0}^T \int_{[a,b]} \int_{0}^t \iint_{\Omega_z} \chi_1(x,y)xK(x,y)f_s(\rmd x)f_s(\rmd y) \rmd s \rmd z \rmd t,
	\end{align}
	\begin{align}
	I_2 = \norm{\psi}_\infty \int_{0}^T \int_{[a,b]}\int_{0}^t \iint_{\Omega_z} (1-\chi_1(x,y))(1-\chi_3(x,y))xK(x,y)f_s(\rmd x)f_s(\rmd y) \rmd s \rmd z \rmd t,
	\end{align}	and
	\begin{align}
	I_3 = \norm{\psi}_\infty \int_{0}^T \int_{[a,b]} \int_{0}^t \iint_{\Omega_z} \chi_3(x,y)xK(x,y)f_s(\rmd x)f_s(\rmd y) \rmd s \rmd z \rmd t.
	\end{align}	

	The corresponding integral of the \(\epsilon\)-problem satisfies the upper bound
	\begin{align}
	\abs{\int_{0}^T\int_{\R_*}\psi(t,z) \int_{0}^t \iint_{\Omega_z}xK(x,y)f^{\epsilon}_s(\rmd x)f^\epsilon_s(\rmd y) \rmd s \rmd z \rmd t} &\leq I_1^\epsilon + I_2^\epsilon + I_3^\epsilon,
	\end{align}
	with
	\begin{align}
	I_1^\epsilon \leq \norm{\psi}_\infty \int_{0}^T \int_{[a,b]} \int_{0}^t \iint_{\Omega_z} \chi_1(x,y)xK(x,y)f^\epsilon_s(\rmd x)f^\epsilon_s(\rmd y) \rmd s \rmd z \rmd t ,
	\end{align}
	\begin{align}
	I_2^\epsilon = \norm{\psi}_\infty \int_{0}^T \int_{[a,b]} \int_{0}^t \iint_{\Omega_z} (1-\chi_1(x,y))(1-\chi_3(x,y))xK(x,y)f^\epsilon_s(\rmd x)f^\epsilon_s(\rmd y) \rmd s \rmd z \rmd t,
	\end{align} 
	and
	\begin{align}
	I_3^\epsilon = 
	\norm{\psi}_\infty \int_{0}^T \int_{[a,b]} \int_{0}^t  \iint_{\Omega_z} \chi_3(x,y)xK(x,y)f^\epsilon_s(\rmd x)f^\epsilon_s(\rmd y) \rmd s \rmd z \rmd t.
	\end{align}
	
	Recall that \(\bar{\epsilon} > 0\) is a small number that we have fixed at the start of the computation. By Proposition \ref{prop:J1_J3}, there exists \(\delta_{\bar{\epsilon}}\) so that for all \(\delta \leq \delta_{\bar{\epsilon}}\)
	\begin{align}
	&\sup_{z > 0}\int_{0}^t \rmd s J^1_{f_s}(z,\delta)  \leq \bar{\epsilon} \nonumber \\
	&\sup_{R>0} \int_{[R/2,R]}\int_{0}^t \rmd s J^3_{f_s}(z,\delta) \rmd z \leq \bar{\epsilon}
	\label{eq:smallness-estimates-flux-terms}
	\end{align}
	and
	\begin{align}
	&\sup_{\epsilon>0}\sup_{z > 0}\int_{0}^t \rmd s J^1_{f^\epsilon_s}(z,\delta)  \leq \bar{\epsilon} \nonumber \\
	&\sup_{\epsilon>0}\sup_{R>0} \int_{[R/2,R]}\int_{0}^t \rmd s J^3_{f^\epsilon_s}(z,\delta) \rmd z \leq \bar{\epsilon},
	\label{eq:smallness-estimates-flux-terms-epsilon}
	\end{align}
	Since the kernel is bounded in \(\cup_{z\in [a,b]}\Omega_z^2(\delta)\), we can use the \wkst-convergence together with the Dominated Convergence Theorem to obtain a cutoff point \(N_0 = N_0(\psi,\delta,\bar{\epsilon}) \in \N\) for which the bound
	\begin{align}
	\norm{\psi}_\infty \int_{0}^T \int_{\R_*} \abs{\int_{0}^t \iint_{\Omega_z} (1-\chi_1)(1-\chi_2)xK(x,y)[f_s(\rmd x)f_s(\rmd y)-f^{\epsilon(n)}_s(\rmd x) f^{\epsilon(n)}_s(\rmd y)]} \leq \bar{\epsilon}
	\label{eq:smallness-middle-term}
	\end{align}
	holds for all \(n \geq N_0\).
	
	Altogether the estimates \eqref{eq:smallness-estimates-flux-terms}, \eqref{eq:smallness-estimates-flux-terms-epsilon}, and  \eqref{eq:smallness-middle-term} imply that for all \(n \geq N_0\), there exists \(B>0\) such that 
	\begin{align}
	\abs{\inner{\psi}{H-H^{\epsilon(n)}}} \leq B\bar{\epsilon}.
	\end{align} Consequently, it follows that
	\begin{align}
		\abs{\inner{\psi}{G-C+H-S}} \leq \tilde{C}\bar{\epsilon}.
	\end{align}
	
	This is true for arbitrarily small \(\bar{\epsilon}>0\), whereby \(\inner{\psi}{G-C+H-S} = 0\). Since the function \(G-C+H-S\) is by the same computations a locally integrable function on \([0,T] \times \R_*\), it implies that the function satisfies \(G-C+H-S = 0\) for almost every pair of points \((t,z) \in [0,T] \times \R_*\). 
	\end{proof}

\section{Properties of the solutions}

In this Section we study the properties of the weak flux solutions as defined in \ref{def:weak-flux-solution}.

\begin{assumption}\label{assumption3}
    Assume that the kernel $K$ and the real parameters $\gamma$ and $\lambda$ satisfy Assumptions \ref{assumption:kernel-regularity} and \ref{assumption:kernel-exponents}. Let  \(f_0 \in \posrad{\R_*}\)  satisfy \(xf_0 \in \posradb{\R_*}\) and
     \(f_t \in C([0,T], \posradb{\R_*})\) be a weak flux solution to \eqref{eq:coagulation-equation} in the sense of Definition \ref{def:weak-flux-solution} with initial data $f_0$.
\end{assumption}
	\begin{proposition}
	\label{prop:time-average-flux-bound}
 Assume that Assumption \ref{assumption3} holds. Then there exists a positive constant \(C_T\) that depends on $\gamma$, $\lambda$, $T$, and the initial data, such that the following time-integrated average moment bound holds for all \(t \in [0,T]\)
	\begin{align}
	\label{ineq:time-average-weighted-R-integral}
	\int_{0}^t \frac{1}{R}\int_{[R/2,R]} R^{\frac{\gamma+3}{2}} f_s(\rmd x) \rmd s \leq C_T.
	\end{align}	
\end{proposition}

In particular, from Proposition \ref{prop:J1_J3} there exists a constant \(\tilde{C}_T\) such that  for every \(\epsilon > 0\), there exists \(\delta_\epsilon >0\), such that for every \(\delta < \delta_\epsilon\), we have
	\begin{align}
	\label{ineq:flux-smallness-region-one-f}
	\sup_{z>0}\int_{0}^t J_{f_t}^1(z;\delta)\rmd s \leq \epsilon \tilde{C}_T
	\end{align}
	and
	\begin{align}
	\label{ineq:flux-smallness-region-three-f}
	\sup_{R > 0} \frac{1}{R}\int_{[R/2,R]} \int_{0}^t J_{f_t}^3(z;\delta) \rmd s \rmd z \leq \epsilon \tilde{C}_T
	\end{align}	

\begin{proof}
Similarly to the proof of Lemma \ref{lemma:power_estimate_f_eps}, we use the equation and the lower bound of the kernel to obtain the two estimates
\begin{align}
	\int_{0}^t \rmd s J_{f_s}(R) \leq t + \int_{(0,R]}x f_0(\rmd x)
\end{align}
and
\begin{align}
	J_{f_s}(R) \geq C' \int_{[R/2,R]} \int_{[R/2,R]} R^{\gamma+1} f_s(\rmd x)f_s(\rmd y) = C' \left(\frac{1}{R} \int_{[R/2,R]} R^{\frac{\gamma+3}{2}} f_s(\rmd x)\right)^2.
	\end{align}
The result then follows from the Cauchy-Schwarz inequality.
\end{proof}

\begin{proposition}
	Assume that Assumption \ref{assumption3} holds. Then the mass is conserved, that is, $M_1(f_t)=M_1(f_0)+t$, for a.e. $t \in [0,T]$. 
\end{proposition}

\begin{proof}
Our strategy is to take the limit in \eqref{evol-eq:weak-flux-solution} as $z \to \infty$.  
To this end, we start by splitting the integration region $\Omega_z$ of the flux $J_{f_s}$ into three regions $\Omega_z^i(\delta) = \Omega_z \cap \Sigma_i(\delta),  i =1,2,3,$ where $\Sigma_i$ is defined in \eqref{eqdef:flux-term-regions} (see also Figure \ref{figure:flux-term-regions}). 
The idea will be to use that the integrals over the first and third region are small according to the estimates \eqref{ineq:flux-smallness-region-one-f} and \eqref{ineq:flux-smallness-region-three-f}, respectively, and to prove that the flux over the second region vanishes. 

However, to be able to apply \eqref{ineq:flux-smallness-region-three-f}, we first need to integrate \eqref{evol-eq:weak-flux-solution} in $z \in [R/2,R]$, which gives
\begin{align}
	\frac{1}{R} \int_{R/2}^R\int_{(0,z]} x f_t(\rmd x)\rmd z - \frac{1}{R} \int_{R/2}^R\int_{(0,z]} x f_0(\rmd x)\rmd z = -\frac{1}{R} \int_{R/2}^R \int_{0}^t J_{f_s}(z) \rmd s \rmd z +  t \frac{1}{R} \int_{R/2}^R \rmd z,	
\end{align} 
for  almost every $t \in [0,T]$ and all $R>0$.
 Using Fubini-Tonelli's Theorem we have that
\begin{align*}
\frac{1}{R} \int_{R/2}^R\int_{(0,z]} x f_t(\rmd x)\rmd z  &= \frac{1}{R} \int_{0}^R\int_{(\max\{ R/2,x\},R]}\rmd z x f_t(\rmd x)\\
& = \frac{1}{R} \left( \int_{0}^{R/2} \frac{R}{2} x f_t(\rmd x) + \int_{R/2}^{R} (R-x) x f_t(\rmd x) \right)\\
& =  \frac{1}{2}\int_{0}^{R}  x f_t(\rmd x) + \int_{R/2}^{R} \left(\frac{1}{2}-\frac{x}{R}\right) x f_t(\rmd x)
\end{align*} 
which converges to $\frac{1}{2}M_1(f_t)$ by  Lebesgue's dominated convergence theorem.

Fix $\varepsilon>0$ arbitrarily. From \eqref{ineq:flux-smallness-region-one-f} and \eqref{ineq:flux-smallness-region-three-f} there is a small  positive $\delta$ such that
\begin{align}\label{eq:J1_J3}
 \frac{1}{R}\int_{[R/2,R]} \int_{0}^t  (J_{f_s}^1(z;\delta)+ J_{f_s}^3(z;\delta)) \rmd s \rmd z \leq \epsilon \tilde{C}_T \quad \text{ for all }\quad  R>0.
\end{align}
On the other hand, using the upper bound of the kernel and the fact that $$\Omega_z^2(\delta) \subset \left[\frac{\delta}{1+\delta}z, \frac{z}{\delta}\right]\times \left[\frac{\delta}{1+\delta}z, z\right],$$  we obtain the following upper bound for \(J_{f_s}^2(z;\delta)\)
\begin{align*}
   J_{f_s}^2(z;\delta) &\leq  C \int_{[\frac{\delta}{1+\delta}z, \frac{z}{\delta} ]}\int_{[\frac{\delta}{1+\delta}z, z]} (x^{1+\gamma+\lambda}y^{-\lambda}+ y^{\gamma+\lambda}x^{1-\lambda})f_s(\rmd x)f_s(\rmd y) \\
   &\leq  C_\delta \int_{[\frac{\delta}{1+\delta}z, \frac{z}{\delta} ]}\int_{[\frac{\delta}{1+\delta}z, z]} x^{\gamma}yf_s(\rmd x)f_s(\rmd y)\\
   &\leq  C_\delta \int_{[\frac{\delta}{1+\delta}z, \infty)}x^{\gamma}f_s(\rmd x) \int_{[\frac{\delta}{1+\delta}z, \infty)} x f_s(\rmd x).
\end{align*}
Since $M_1(f_s)<\infty $, for all $s \in [0,T]$, and $\gamma<1$, there is a large enough $z_*$, depending on $\varepsilon$ and $\delta$, such that, for all $z>z_*$,
$  J_{f_s}^2(z;\delta) \leq \varepsilon  C_T $. 
We now integrate with respect to $s \in [0,t]$ and $z \in [\frac{R}{2},R]$, for each  $R>2z_*$. Putting this together with \eqref{eq:J1_J3}, we finally obtain that 
 for each $\varepsilon>0$, there is a large enough $z_*$ such that 
\begin{align*}
\frac{1}{R} \int_{R/2}^R \int_{0}^t J_{f_s}(z) \rmd s \rmd z \leq \varepsilon C_T, \quad \text{for all } R>2z_*, 
\end{align*}
which implies the result.
\end{proof}

\begin{proposition}
		Assume that Assumption \ref{assumption3} holds. Let \(x_0 \in \R_*\). Then, for all \(t \in [0,T]\), the following estimate is satisfied:
    \begin{align}\label{eq:estimate_x0}
	   \int_{0}^t \int_{(0,x_0]} x f_s(\rmd x)\rmd s \leq C_T x_0^{\frac{1-\gamma}{2}}.
	\end{align}
\end{proposition}
\begin{proof}
The result can be obtained by decomposing the integration domain $(0,x_0]$ using dyadics $(0,x_0] = \cup_{n =0}^\infty (2^{-(n+1)}x_0,2^{-n}x_0]$ and by using estimate \eqref{ineq:time-average-weighted-R-integral} in each subregion. The proof follows exactly the same steps as the proof of Lemma \ref{lemma:f_eps_near0}. 
\end{proof}

%{\rd we could probably get the lower bound as well?! which would correspond to $f \sim x^{-\frac{\gamma+3}{2}}$ as $x \to 0$!}

\begin{remark}
In particular, \eqref{eq:estimate_x0} implies that no solution to the flux equation accumulates mass at zero. The physical interpretation is that the dust entering the system is instantaneously converted into particles with  positive size and no dust remains in the system at any time.   
\end{remark}
	
\begin{proposition}
	Assume that Assumption \ref{assumption3} holds. Then $f$ satisfies the weak coagulation equation
\begin{align}
&\int_{(0,\infty)} x \varphi(t,x) f_t(\rmd x) =  \int_{(0,\infty)} x \varphi(0,x) f_0(\rmd x) + \int_0^t \int_{(0,\infty)} x \partial_s\varphi(s,x) f_s(\rmd x)\rmd s\\
& + \frac{1}{2}\int_0^t \int_{(0,\infty)}\int_{(0,\infty)}K(x,y) [(x+y)\varphi(s,x+y)-x\varphi(s,x)-y\varphi(s,x)]f_s(\rmd x)f_s(\rmd y)\rmd s  
\end{align}
for every $\varphi \in C^1_c([0,T]\times \R_*)$ and almost every $t \in [0,T]$, together with the flux boundary condition
\begin{equation} \label{eq:boundary_cond}
 \int_0^t J_{f_s}(z) \rmd s \to t \quad \text{ as } \quad z \to 0, \quad \text{ a.e. } \quad t \in [0,T].
 \end{equation} 
\end{proposition}	
\begin{proof}
The proof follows from the weak formulation given in \eqref{eq:test-function-formulation-limit-flux}. 
Take $\psi \in C_c([0,T]\times \R_*)$ such that $\int_{\R_*} \psi(t,z)\rmd z = 0$ pointwise in $t\in [0,T]$. Then using  Fubbini-Tonelli's Theorem and a symmetrization argument we get from \eqref{eq:test-function-formulation-limit-flux} that
\begin{align*}
&\int_0^T \left(\int_{\R_*} \varphi(t,x) xf_t(\rmd x) - \int_{\R_*} \varphi(0,x) xf_0(\rmd x) \right.\\
&\left. +\frac{1}{2}\int_0^t \int_{\R_*}\int_{\R_*} 
\left[ \varphi(x)x + \varphi(y)y - \varphi(x+y)(x+y) \right] K(x,y) f_s(\rmd x)f_s(\rmd y)\rmd s \right) \rmd t = 0  
\end{align*}
 where $\varphi(t,x) := \int_x^\infty \psi(t,z)\rmd z$ satisfies $\varphi \in  C_c^1([0,T]\times \R_*)$.

As for the boundary condition, for any sequence $(z_k)_k$ such that $z_k \to 0$ as $k \to \infty$, we have from Fatou's Lemma that   
$$0 \leq \int_0^t \liminf_{k\to \infty} \int_{(0,z_k]} xf_s(\rmd x) \rmd s \leq \liminf_{k\to \infty} \int_0^t \int_{(0,z_k]} xf_s(\rmd x) \rmd s$$
and the right hand side is zero by  \eqref{eq:estimate_x0}. This implies   \eqref{eq:boundary_cond}.

\end{proof}

%{\rd Check if we can prove that $t \mapsto \int_{(0,z]} x \varphi(x,t)f_t(dx)$ is continuous, check if we can remove the almost every $t$}

	\section{Convergence to the steady state for the constant kernel}\label{sec:constantKernel}
	In this section, we study a special case of the flux solution, namely the one where the coagulation kernel is constant. In this case, the coagulation kernel has such a simple form that transformation techniques allow us to construct more explicitly the time-dependent measure that solves the flux from zero equation, with the additional restriction that the initial data is zero. This corresponds to starting the system from a state where there are no particles of any size. We are assuming that the kernel satisfies \(K \equiv 2\) on \(\R_*^2\). A closely related problem, the constant kernel coagulation equation \emph{without a source term} but with nonzero initial data was studied by Menon and Pego in \cite{menon_approach_2004}.
	
	The coagulation equation with constant kernel has been studied extensively, since it allows for a particularly simple solution via the Bernstein transform in the continuous case or the method of moments in the discrete case. We will be dealing with the continuous case. For basic facts about the Bernstein transform and Bernstein functions, see \cite{schilling_bernstein_2012}. The central point to note is that if a we have a measure \(\mu \in \posrad{\R_*}\) that additionally satisfies the moment condition that the function \(\min(1,x)\) is integrable agains the measure, then its \emph{Bernstein transform} is defined as the function \(\bern{\mu} \colon \R_* \to \R_+\), whose value at any given point \(\lambda \in \R_*\) is given by
	\begin{align}
	\bern{\mu}(\lambda) = \int_{\R_*} \left(1-\rme^{-\lambda x}\right) \mu(\rmd x).
	\end{align}
	
	The Bernstein transform of this kind of measure determines what is sometimes called a \emph{Bernstein function}. A function \(g \colon \R_* \to \R\) is a Bernstein function, if it is non-negative and if its first derivative \(\partial_\lambda g\) is a completely monotone function in the sense that \((-1)^{n} \partial^n_\lambda \left(\partial_\lambda g\right) \geq 0\) for all \(n \in \N_0\). The Lévy--Khintchine representation states that to each such function, there corresponds a measure is \(\posrad{\R_*}\) such that \(\min(1,x)\) is integrable with respect to it. More precisely, the theorem states that there exist positive constants \(a,b\) and a measure \(\mu\) such that for all \(\lambda >0\)
	\begin{align}
	g(\lambda) = a + b\lambda + \int_{\R_*}(1-\rme^{-\lambda x})\mu(\rmd x).
	\end{align}
	This representation, whose proof can be found, for example, in the monograph on Bernstein functions by Schilling et al. \cite[Theorem 3.2]{schilling_bernstein_2012}.
	
	Let us move from discussing techniques towards the formulating problem statement. To this end, consider the following family of differential equations, which makes sense for suitable test functions \(\phi\):
	\begin{align}
	\begin{cases}
	\dv{t}\inner{\phi}{xf_t} &= \inner{(x\phi)}{\K[f_t]} + \inner{\phi}{\delta_0}, \quad t > 0 \\
	\inner{\phi}{xf_0} = 0 
	\end{cases}	
	\end{align}	
	
	Here the coagulation operator is formally defined as
	\begin{align}
	\inner{(x\phi)}{\K[f_t]} &= \frac{1}{2}\int_{\R_*}\int_{\R_*}\left((x+y)\phi(x+y)-x\phi(x)-y\phi(y)\right)f_t(\rmd x)f_t(\rmd y).
	\end{align}
	To guarantee that the integral on the right hand side is well defined and is well defined, we require that \(\phi \in C^1_c(\R_+)\). The corresponding time-integrated version, which likewise makes sense when \(\phi \in C^1_c(\R_+)\), is
	\begin{align}
	\label{eq:time-integrated-constant-kernel-flux-eq}
	\inner{\phi}{xf_t} = \int_{0}^t \inner{(x\phi)}{\K[f_s]} \rmd s + t\inner{\phi}{\delta_0}, \quad t > 0	
	\end{align}
	
	\begin{definition}
	\label{def:constant-kernel-flux-solution}
	We say that a time-dependent measure \(f \in C([0,+\infty),\posrad{\R_*})\) satisfies the constant kernel coagulation equation with a flux from zero, in case \(xf \in C([0,+\infty), \posradb{\R_*})\) and if \eqref{eq:time-integrated-constant-kernel-flux-eq} is satisfied for all \(\phi \in C_c^1(\R_+)\)
	\end{definition}
	
	We will construct a solution to \eqref{eq:time-integrated-constant-kernel-flux-eq} in the sense of Definition \ref{def:constant-kernel-flux-solution} by first considering the following family of initial value problems
	\begin{align}
	\begin{cases}
	\dv{t} \inner{\phi}{f_t^\epsilon} = \inner{\phi}{\K[f^\epsilon_t]} + \frac{1}{\epsilon}\phi(\epsilon) \\
	f_0 = 0.
	\end{cases}
	\end{align}
	for \(\phi \in C_c(\R_+)\). Here the coagulation operator is defined as
	\begin{align}
	\inner{\phi}{\K[\mu]} = \int_{\R_*}\int_{\R_*} D[\phi](x,y)\mu(\rmd x)\mu(\rmd y).
	\end{align}
	If \(\mu\) is supported inside \([\epsilon,+\infty)\) and if it has a finite first moment, this is well-defined.
	
	If this problem is formulated in the integrated form, it states that for all \(\phi \in C_c(\R_*)\), we have
	\begin{align}
	\label{eq:constant-kernel-epsilon}
	\inner{\phi}{f_t^\epsilon} = \int_{0}^t \inner{\phi}{\K[f_s^\epsilon]}\rmd s + t \frac{1}{\epsilon}\phi(\epsilon).
	\end{align}
	\begin{definition}
	\label{def:constant-kernel-coagulation-equation-solution}
	We say that \(f \in C([0,\infty), \posradb{\R_*})\) satisfies the constant kernel coagulation equation with a compactly supported source, in case \(\spt(f) \subset [a,+\infty)\), with \(a>0\), if \(\sup_{t\in [0,T]} M_1(f_t) <+\infty\) for all \(T\), and if \eqref{eq:constant-kernel-epsilon} is satisfied for all \(C_c(\R_+)\).
	\end{definition}
	
	If \(f_t^\epsilon\) is a solution to in the sense of Definition \ref{def:constant-kernel-coagulation-equation-solution}, we may use the function \((1-\rme^{-\lambda x})\) as a test function to get the following evolution equation for the Bernstein transform of \(f_t^\epsilon\):
	\begin{align}
	\label{eq:epsilon-bernstein-gronwall}
	\bern{f_t^\epsilon}(\lambda) = -\int_{0}^t\bern{f_s^\epsilon}(\lambda)^2 \rmd s + t \frac{1}{\epsilon}\left(1-\rme^{-\lambda \epsilon}\right).
	\end{align}
	
	Equation \eqref{eq:epsilon-bernstein-gronwall} has a unique time-continuous solution, given by
	\begin{align}
	\label{eq:epsilon-bernstein-identity}
	\bern{f_t^\epsilon}(\lambda) = \sqrt{\frac{1}{\epsilon}\left(1-\rme^{-\lambda \epsilon}\right)}\tanh\left(\sqrt{\left(\frac{1}{\epsilon}(1-\rme^{-\lambda\epsilon})\right)}t\right).
	\end{align}
		The function \(\lambda \mapsto \frac{1}{\epsilon}\left(1-\rme^{-\lambda \epsilon}\right)\) is a Bernstein function. In order to conclude that \(\bern{f_t^\epsilon}\) is a Bernstein function for all times -- and thus corresponds to a well--behaved measure on the particle size distribution side, it is enough to prove the following lemma.

	\begin{lemma}
	\label{lemma:bernstein-limit}
	For every \(t > 0\), the function \(h_t \colon \R_* \to \R\), defined by
	\begin{align}
	h_t(\lambda) = \sqrt{\lambda} \tanh(\sqrt{\lambda}t), \quad \lambda \in \R_*
	\end{align}
	is a Bernstein function.
	\end{lemma}
	
	\begin{proof}
	We can first consider the function \(h \coloneqq h_1\). After showing that this is a Bernstein function, the theory of Bernstein functions can be used to show that so is any \(h_t\) with \(t >0\).
	 	
	The function \(h\) is a positive function. Being a product of two infinitely differentiable functions, it also has  derivatives of all order. To conclude that the function is also Bernstein, we still have to show that its first derivative is a completely monotone function.
	
	A direct computation gives
	\begin{align}
	\partial_\lambda h(\lambda)	&= \frac{1}{2}\lambda^{-1/2}\tanh(\sqrt{\lambda}) + \frac{1}{2}\left(1-\tanh^2(\sqrt{\lambda})\right).
	\end{align}
	
	Now, define the following functions
	\begin{align}
	&g_1(\lambda) = \lambda^{-1/2}\tanh(\sqrt{\lambda}) \nonumber \\
	&g_2(\lambda) = \left(1-\tanh^2(\sqrt{\lambda})\right) \nonumber \\
	&\delta(\lambda) \coloneqq g_1(\lambda)-g_2(\lambda) \nonumber \\
	&g(\lambda) \coloneqq g_1(\lambda) + g_2(\lambda) = 2\partial_\lambda h(\lambda).
	\end{align}
	
	Each of these is a positive function. For \(g_1,g_2\) and \(g\) this is obvious, and for \(\delta\) it follows from analyzing the function \(x \mapsto \frac{\tanh(x)}{x} - (1-\tanh^2(x))\) on \([0,+\infty)\).
	
	We now consider the derivatives of these functions. The function \(g_1\) satisfies
	\begin{align}
	\partial_\lambda g_1(\lambda) &= -\frac{1}{2}\lambda^{-3/2}\tanh(\sqrt{\lambda}) + \lambda^{-1/2}\frac{1}{2}\lambda^{-1/2}(1-\tanh^2(\sqrt{\lambda})) \nonumber \\
	&=-\frac{1}{2}\lambda^{-1}\left(\lambda^{-1/2}\tanh(\sqrt{\lambda}) - (1-\tanh^2(\sqrt{\lambda}))\right) \nonumber \\
	&=-\frac{1}{2}\lambda^{-1}\delta(\lambda).
	\end{align}
	
	The function \(g_2\) satisfies 
	\begin{align}
	\partial g_2(\lambda) &= -g_1(\lambda)g_2(\lambda).
	\end{align}
	
	The function \(g\), which is the sum of these two functions, satisfies
	\begin{align}
	\partial g(\lambda) = -\frac{1}{2}\lambda^{-1} \delta(\lambda) - g_1(\lambda)g_2(\lambda).
	\end{align}
	
	Finally, the derivative of the difference function \(\delta\) is given by
	\begin{align}
	\partial \delta(\lambda) &= -\frac{1}{2}\lambda^{-1}\delta(\lambda) + g_1(\lambda)g_2(\lambda) \nonumber \\
	&= -\frac{1}{2}\lambda^{-1}\delta(\lambda) - g_1(\lambda)^2 - g_2(\lambda)^2 - g(\lambda)^2.
	\end{align}
	
	The function \(\lambda \mapsto \lambda^{-1}\) is a completely monotone function. We will now prove by induction, that \((-1)^{n}\partial^n f(\lambda) \geq 0\) for all \(n \in \N\) and \(\lambda \in \R_*\), when \(f \in \{g_1,g_2,g,\delta\}\). The base-case is clearly true. So assume that the claim is true for all \(k \leq n\). Let us now consider the case \(n+1\).
	
	By induction assumption, the following computations hold.
	\begin{align}
	(-1)^{n+1}\partial^{n+1} g_1(\lambda) &= \frac{1}{2}(-1)^n \partial^n \left((-1)^2 \lambda^{-1}\delta(\lambda)\right) \nonumber \\
	&= \frac{1}{2}\sum_{k=0}^n \binom{n}{k} \left((-1)^k\partial^{k} \lambda^{-1}\right)\left((-1)^{n-k}\partial^{n-k}\delta(\lambda)\right) \geq 0.
	\end{align}
	
	Similarly,
	\begin{align}
	(-1)^{n+1}\partial^{n+1} g_2(\lambda) &= \sum_{k=0}^n \binom{n}{k} \left((-1)^k\partial^{k} g_1(\lambda)\right)\left((-1)^{n-k}\partial^{n-k}g_2(\lambda) \right) \geq 0
	\end{align}
	and thus also \((-1)^{n+1}\partial^{n+1}g(\lambda) \geq 0\). Finally,
	\begin{align}
	(-1)^{n+1}\partial^{n+1}\delta(\lambda) &=  (-1)^n \partial^{n}\left(\frac{1}{2}\lambda^{-1}\delta(\lambda) + g_1(\lambda)^2 + g_2(\lambda)^2 + g(\lambda)^2\right) \geq 0
	\end{align}
	by a similar argument. It follows by induction that each one of \(f \in \{g_1,g_2,\delta, g\}\) satisfies \((-1)^n \partial_\lambda^n f(\lambda) \geq 0\) for all \(\lambda\), all \(n\). In particular, this shows that for all \(n\) and \(\lambda \in \R_*\),
	\begin{align}
	(-1)^n \partial^n g(\lambda) \geq 0.
	\end{align}
	
	The function \(h(\lambda) = \sqrt{\lambda}\tanh(\sqrt{\lambda})\) therefore is a Bernstein function, and by a simple scaling argument so are also each of the functions
	\begin{align}
	h_t(\lambda) = \sqrt{\lambda}\tanh(\sqrt{\lambda}t)
	\end{align}
	is one too. Indeed, we can write
	\begin{align}
	h_t(\lambda) = \frac{1}{t} (h \circ \alpha_t)(\lambda),
	\end{align}
	with \(\alpha_t(\lambda) = t^2 \lambda\). Since scaling by a fixed constant and the identity function \(\lambda \to \lambda\) are both Bernstein functions, and since the property of being a Bernstein function is preserved under composition, it follows that \(h_t\) is a Bernstein function for any \(t > 0\).
	\end{proof}	
	
	Since \(\bern{f_t^\epsilon}\) is a Bernstein function and \eqref{eq:constant-kernel-epsilon} is satisfied by all functions of the form \(x \mapsto (1-\rme^{-\lambda x})\) as well as their linear combinations, it follows by the \(C([0,+\infty])\)-density of this linear span that the equation \eqref{eq:constant-kernel-epsilon} is satisfied for all test functions \(\phi \in C_c(\R_+)\). Thus, we have a solution to \eqref{eq:constant-kernel-epsilon}. The uniqueness of this solution follows, since the Bernstein transform of a candidate solution is a \(C^1\) function in time.
	
	\begin{lemma}
	\label{lemma:linear-increase-mass-epsilon}
	The measure \(f_t^\epsilon\), which solves Equation \eqref{eq:constant-kernel-epsilon} in the sense of Definition \ref{def:constant-kernel-coagulation-equation-solution}, has linearly increasing mass:
	\begin{align}
	M_1(f_t^\epsilon) = t.
	\end{align}
	\end{lemma}
	\begin{proof}
	For each \(\delta \in (0,1/2)\), let us define a function \(\eta_\delta \in C_c(\R_*)\) piecewise as follows: 
	\begin{align}
	\eta_\delta(x) = \begin{cases} 0, &\quad x \in (0, \delta) \\
	 \frac{1}{2\delta} x ,&\quad x \in [\delta, 2\delta) \\
	 1, &\quad x \in [2\delta, 1/\delta) \\
	 \delta(\frac{2}{\delta}-x), &\quad x \in [1/\delta, 2/\delta) \\
	 0, &\quad x \in [2/\delta,+\infty).
	\end{cases}
	\end{align}
	
	We have \(x\eta_\delta \in C_c(\R_*)\), whereby \(x \eta_\delta\) is a function for which the weak formulation \eqref{eq:constant-kernel-epsilon} makes sense. This tells us that
	\begin{align}
	\int_{\R_*} x\eta_\delta(x) f_t^\epsilon(\rmd x) = \int_{0}^t \int_{\R_*}\int_{\R_*} D[x\eta_\delta](x,y)f_t^\epsilon(\rmd x)f_t^\epsilon(\rmd y) + t\eta_\delta(\epsilon).
	\end{align}
	Pointwise, we have \(x\eta_\delta(x) \to x\) for all \(x \in \R_*\), and \(D[x\eta_\delta](x,y) \to 0\) for all \(x,y \in \R_*\). By the Dominated Convergence Theorem, it follows that
	\begin{align}
	\int_{\R_*} xf_t^\epsilon(\rmd x) = t.
	\end{align}
	\end{proof}
	
	\begin{lemma}
	\label{lemma:lk-representation-limit}
	For every \(t \in [0,+\infty)\), the Lévy triplet \((a_t,b_t,f_t)\) corresponding to the Bernstein function \(h_t(\lambda) = \sqrt{\lambda}\tanh(\sqrt{\lambda}t)\) satisfies \(a_t=b_t=0\). In other words, the corresponding unique measure \(f_t\) satisfies
	\begin{align}
	\sqrt{\lambda} \tanh(\sqrt{\lambda}t) = \int_{\R_*} (1-\rme^{-\lambda x})f_t(\rmd x)
	\end{align}
	\end{lemma}
	\begin{proof}
	By the Lévy--Khintchine representation of \(h_t\), we know that there exists a unique triple \((a_t,b_t,f_t)\) corresponding to \(h_t\) and the following relation holds
	\begin{align}
	\label{eq:limits-levy-representation}
	\sqrt{\lambda}\tanh(\sqrt{\lambda}t) = a_t + b_t\lambda + \int_{\R_*}(1-\rme^{-\lambda x})f_t(\rmd x),
	\end{align}
	where \(\int_{\R_*}\min(1,x)f_t(\rmd x) <+\infty\).
	
	Using the Dominated Convergence Theorem, we may compute the limits of the left and right hand sides of \eqref{eq:limits-levy-representation} to obtain
	\begin{align}
	\lim_{\lambda \to 0} \sqrt{\lambda} \tanh(\sqrt{\lambda} t) = 0
\end{align}
	and
	\begin{align}
	\lim_{\lambda \to 0} \left(a_t + b_t \lambda + \int_{\R_*}(1-\rme^{-\lambda x})f_t(\rmd x)\right) = a_t,
	\end{align}		 	
	whereby the first claim follows. Thus, we have
	\begin{align}
	\label{eq:levy-k-repre-b}
	\sqrt{\lambda}\tanh(\sqrt{\lambda}t) = b_t\lambda + \int_{\R_*}(1-\rme^{-\lambda x})f_t(\rmd x).
	\end{align}

	Let us now move on to determine \(b_t\). In the Lévy--Khintchine representation, \(b_t\) is determined through the derivative of \(h_t\) as follows. Consider the derivative of \(h_t\) with respect to \(\lambda\):
	\begin{align}
	h'_t(\lambda) \coloneqq \partial_\lambda h_t(\lambda) = \frac{1}{2}\lambda^{-1/2}\tanh(\sqrt{\lambda} t) + \frac{1}{2}t\left(1-\tanh^2(\sqrt{\lambda} t)\right).
	\end{align}
	Since this is a completely monotone function, it follows by Bernstein's theorem for the Laplace transform (see \cite[Theorem 1.4]{schilling_bernstein_2012}) that there exists a unique (positive) measure \(\nu_t\) such that
	\begin{align}
	h'_t(\lambda) = \int_{[0,+\infty)} \rme^{-\lambda x}\nu_t(\rmd x).
	\end{align}
	
	The constant \(b_t\) in \eqref{eq:levy-k-repre-b} is related to this measure as \(b_t = \nu_t(\{0\})\geq 0\). Thus, we need to show that the set \(\{0\}\) is of measure zero with respect to \(\nu_t\).
	
	To this end, consider the sequence of functions \(\psi_n \colon \R_+\to \R\), given by \(\psi_n(x) = \rme^{-nx}\). We can define a new function pointwise as \(\psi(x) = \liminf_{n\to\infty} \psi_n(x)\), and this function turns out to be the characteristic function of the set \(\{0\}\), so that \(\psi(x) = \cf{x \in \{0\}}\). Using Fatou's lemma, we obtain the following upper bound for \(b_t\):
	\begin{align}
	\label{ineq:bt_upper-bound}
	b_t = \nu_t(\{0\}) = \int_{\R_+} \psi(x) \nu_t(\rmd x) \leq \liminf_{n\to\infty} \int_{\R_+} \rme^{-n x}\nu_t(\rmd x).
	\end{align}
	But since
	\begin{align}
	\int_{\R_+} \rme^{-n x} \nu_t(\rmd x) = h'_t(n) = \frac{1}{2n^{1/2}} \tanh(n^{1/2}t) + \frac{1}{2}t\left(1-\tanh^2(n^{1/2}t)\right),
	\end{align}
	we have
	\begin{align}
	\label{eq:fatou-laplace}
	\liminf_{n\to\infty} \int_{\R_+} \rme^{-nx}\nu_t(\rmd x) = 0.
	\end{align}
	Combining the lower bound \(b_t \geq 0\) -- coming from the positivity of the measure \(\nu_t\) -- together with the upper bound \(b_t\leq 0\) -- coming from \eqref{ineq:bt_upper-bound} and \eqref{eq:fatou-laplace} -- we obtain the desired result \(b_t = 0\).
	\end{proof}
	
	\begin{lemma}
	Let \(f \colon [0,+\infty) \to \posrad{\R_*}\) be the time-dependent measure, defined pointwise in time so that each \(f_t\) is the unique measure that satisfies
	\begin{align}
	h_t(\lambda) = \int_{\R_*}(1-\rme^{-\lambda x})f_t(\rmd x).
	\end{align}
	Then \(f\) is continous in time in the sense that for every \(\phi \in C_c(\R_*)\), the mapping
	\begin{align}
	t \mapsto \inner{\phi}{f_t}
	\end{align}
	is continuous.
	\end{lemma}
	\begin{proof}
	Since for every \(\lambda > 0\), the function \(t \mapsto h_t(\lambda)\) is continuous, it follows by the basic theory of Bernstein functions (see \cite[Corollary 3.9]{schilling_bernstein_2012}) that \(f_s \to f_t \) in the \wkst-topology, so that for \(\phi \in C_c(\R_*)\), we have
	\begin{align}
	\inner{\phi}{f_s} \to \inner{\phi}{f_t}, \quad \text{ as } s\to t.
	\end{align}
	Thus, for each fixed test function \(\phi \in C_c(\R_*)\), the function \(
	t \mapsto \inner{\phi}{f_t}\)
	is continuous.	
	\end{proof}
	We use the notation \(f \in C([0,+\infty), \posrad{\R_*})\). This should be considered in the sense of belonging to a set and not as being an element of a normed space.
	\begin{proposition}
	\label{prop:linear-increase-mass}	
	Let \(f \in C([0,+\infty), \posrad{\R_*})\) be the time-dependent measure such that for each fixed time \(f_t\) is the unique positive Radon measure corresponding to the Bernstein function \(h_t(\lambda) = \sqrt{\lambda} \tanh(\sqrt{\lambda} t)\) through the identity 
	\begin{align}
	\sqrt{\lambda}\tanh(\sqrt{\lambda} t) = \int_{\R_*}(1-\rme^{-\lambda x}) f_t(\rmd x),
	\end{align} whose existence is guaranteed by Lemma \ref{lemma:lk-representation-limit}. Then its mass is linearly increasing. Indeed, at any time \(t \in [0,+\infty)\), we have
	\begin{align}
	\int_{\R_*}xf_t(\rmd x) = t.
	\end{align}
	\end{proposition}
	\begin{proof}
	Considering the following collection of functions, defined for each \(n \in \N\) as
	\begin{align}
	\phi^\lambda_n(x) = \frac{\left(1-\rme^{-(\lambda + \frac{1}{n})x}\right) - \left(1-\rme^{-\lambda x}\right)}{\frac{1}{n}} = n(\rme^{-\lambda x}- \rme^{-(\lambda + \frac{1}{n})x}).
	\end{align}
	This sequence of functions is monotonely increasing. Furthermore, we have \(\phi_n^\lambda(x)\to \rme^{-\lambda x}x\) for every \(x\). It therefore from the Monotone Convergence Theorem that
	\begin{align}
	\frac{1}{n}\left(h_t(\lambda + \frac{1}{n})-h_t(\lambda)\right) = \int_{\R_*} \phi_n^\lambda(x) f_t(\rmd x) \to \int_{\R_*} x\rme^{-\lambda x}f_t(\rmd x).
	\end{align}
	
	Since the left hand side converges to the derivative of \(h_t\) at the point \(\lambda \in \R_*\), it follows that the derivative satisfies
	\begin{align}
	h'_t(\lambda) = \int_{\R_*} x\rme^{-\lambda x}f_t(\rmd x).
	\end{align}
	On the other hand, we can compute \(h_t'(\lambda)\) explicitly, and thus we have the following identity for the Laplace transform of \(xf_t(\rmd x)\):
	\begin{align}
	\int_{\R_*} x \rme^{-\lambda x} f_t(\rmd x) = \frac{1}{2}\lambda^{-1/2}\tanh(\sqrt{\lambda}t) + \frac{1}{2}t\left(1-\tanh^2(\sqrt{\lambda}t)\right).
	\end{align}
	
	We may conclude by yet another application of the Monotone Convergence Theorem that
	\begin{align}
	\int_{\R_*} x f_t(\rmd x) &= \lim_{\lambda \to 0}  \int_{\R_*}\rme^{-\lambda x} x f_t(\rmd x) \nonumber \\
	&= \lim_{\lambda \to 0} \left(\frac{1}{2}\lambda^{-1/2}\tanh(\sqrt{\lambda}t) + \frac{1}{2}t\left(1-\tanh^2(\sqrt{\lambda}t)\right)\right) \nonumber \\
	&= \frac{1}{2}t + \frac{1}{2}t = t.
	\end{align}
	\end{proof}
	
	Let \(F^\epsilon_t\) and \(F_t\) be the unique Radon measures satisfying for all \(\phi \in C_c(\R_*)\)
	\begin{align}
	\inner{\phi}{F_t^\epsilon} = \int_{\R_*} \phi(x)xf^\epsilon_t(\rmd x)
	\end{align} 
	and
	\begin{align}
	\inner{\phi}{F_t} = \int_{\R_*} \phi(x)xf_t(\rmd x).
	\end{align}
	Together, Lemma \ref{lemma:linear-increase-mass-epsilon} and Proposition \ref{prop:linear-increase-mass} show that the \wkst-convergence of the measures \(F_t^\epsilon\) towards \(F_t\) and the measures \(F_t^\epsilon \otimes F_t^\epsilon\) towards \(F_t \otimes F_t\) can be improved to the following stronger notions of convergence.
	\begin{align}
	\inner{\phi}{F_t^\epsilon} \to \inner{\phi}{F_t}, \quad \phi \in C_b(\R_*)
	\end{align}
	and
	\begin{align}
	\inner{\phi}{F_t^\epsilon \otimes F_t^\epsilon} \to \inner{\phi}{F_t \otimes F_t}, \quad \phi \in C_b(\R_*^2).
	\end{align}
	The difference here being that the admissible sets of test functions are all of \(C_b(\R_*)\) or \(C_b(\R_*^2)\) instead of their subsets \(C_c(\R_*)\) or \(C_c(\R_*^2)\). This allows us to see the boundary condition.
		
	With this information, we are ready to prove the main proposition of this section.
	
	\begin{proposition}
	\label{prop:existence-constant-kernel}
		The time-dependent measure \(f \in C([0,+\infty), \posrad{\R_*})\), determined by its correspondence with the Bernstein function \(h_t(\lambda) = \sqrt{\lambda}\tanh(\sqrt{\lambda}t)\), determines the unique solution to \eqref{eq:time-integrated-constant-kernel-flux-eq}.
	\end{proposition}
	\begin{proof}
	For every \(\epsilon > 0\) and every \(t \in \R_+\), the function 
	\begin{align}
	\label{eq:bern-epsilon}
	\bern{f_t^\epsilon}(\lambda) = \sqrt{\frac{1}{\epsilon}\left(1-\rme^{-\lambda \epsilon}\right)}\tanh\left(\sqrt{\frac{1}{\epsilon}(1-\rme^{-\lambda \epsilon})}t\right)
	\end{align} is a Bernstein function. This follows from Lemma \ref{lemma:bernstein-limit}, which proves that \(\lambda \mapsto \sqrt{\lambda} \tanh(\sqrt{\lambda}t)\), when combined with the fact that \(\lambda \mapsto \frac{1}{\epsilon}(1-\rme^{-\lambda \epsilon})\) is one too and that the composition of two Bernstein functions is again a Bernstein function. On the other hand, the functions in \eqref{eq:bern-epsilon} converge pointiwse to \(h_t(\lambda) = \sqrt{\lambda} \tanh(\sqrt{\lambda}t)\). The corresponding unique measures \(f_t^\epsilon\) therefore converge in the sense of \wkst-convergence to the unique measure Radon measure \(f_t\) that corresponds to \(h_t(\lambda)\) through the Lévy--Khintchine representation. It now remains to show that the measure \(f_t\) solves Equation \eqref{eq:time-integrated-constant-kernel-flux-eq}. The mass of the measures \(f_t\) and \(f_t^\epsilon\) increase linearly in time and they match for all times, so the first moment is finite for all times. 
	
	Let \(\phi \in C_c^1(\R_+)\). We can use the test-function \(\psi(x) = x\phi(x) \in C_c^1(\R_+)\) in \eqref{eq:constant-kernel-epsilon} to obtain
	\begin{align}
	\label{eq:flux-with-epsilon-and-smooth}
	\inner{\phi}{xf_t^\epsilon} = \frac{1}{2}\int_{\R_*}\int_{\R_*}\left((x+y)\phi(x+y)-x\phi(x)-y\phi(y)\right)f^\epsilon_t(\rmd x)f^\epsilon_t(\rmd y) + t\phi(\epsilon).
	\end{align}
	
	Let \(\bar{\epsilon} > 0\). Since
	\(\inner{\bar{\phi}}{F_t^\epsilon} \to \inner{\bar{\phi}}{F_t}\) for all \(\bar{\phi} \in C_b(\R_*)\) and since the restriction of a \(C_c(\R_+)\) function to \(\R_*\) is a function in \(C_b(\R_*)\), it follows that there exists some \(\delta_1 >0\) such that
	\begin{align}
	\abs{\inner{\phi}{F_t^\epsilon}-\inner{\phi}{F_t}} \leq \bar{\epsilon},
	\end{align}
	whenever \(0<\epsilon < \delta_1\).
	
	Clearly by the continuity of \(\phi\), there exists \(\delta_2 > 0\) such that \(\abs{t\phi(\epsilon)-t\phi(0)} <\bar{\epsilon}\), whenever \(0<\epsilon<\delta_2\).
	
	Finally, by the Dominated Convergence Theorem with respect to the time-integral, finiteness of the first moments for all \(f_t^\epsilon\) and \(f_t\), together with the fact that \(\inner{\bar{\phi}}{F^\epsilon_t \otimes F^\epsilon_t} \to \inner{\bar{\phi}}{F_t \otimes f_t}\) for all \(\phi \in C_b(\R_*^2)\), it follows that there exists \(\delta_3 > 0\) such that
	\begin{align}
	\abs{\frac{1}{2}\int_{0}^t\int_{\R_*}\int_{\R_*}\left((x+y)\phi(x+y)-x\phi(x)-y\phi(y)\right)\left(xf_t^\epsilon(\rmd x)yf_t^\epsilon(\rmd y)-xf_t(\rmd x)yf_t(\rmd y)\right)} < \bar{\epsilon},
	\end{align}
	whenever \(0 < \epsilon < \delta_3\).
	
	Indeed, for any \(L\in \R_*\) with \(L > \spt(\phi)\), the following decomposition of the integral can be performed
	\begin{align}
	&\int_{\R_*}\int_{\R_*} \left((x+y)\phi(x+y)-x\phi(x)-y\phi(y)\right)\left(f_t^\epsilon(\rmd x)f_t^\epsilon(\rmd y)-f_t(\rmd x)f_t(\rmd y)\right) \\
	&= \int_{(0,L]}\int_{(0,L]}\left((x+y)\phi(x+y)-x\phi(x)-y\phi(y)\right)\left(f_t^\epsilon(\rmd x)f_t^\epsilon(\rmd y)-f_t(\rmd x)f_t(\rmd y)\right) \\
	& -2 \int_{[0,L]} \int_{[L,+\infty)} x\phi(x)  \left(f_t^\epsilon(\rmd x)f_t^\epsilon(\rmd y)-f_t(\rmd x)f_t(\rmd y)\right).
	\end{align}
	The first term converges to zero, since the continuous differentiability of \(\phi\) implies that
	\begin{align}
	(x,y) \mapsto \frac{(x+y)\phi(x+y)-x\phi(x)-y\phi(y)}{xy}
	\end{align}	 is a \(C_b(\R_*^2)\) function and since we can push the pointwise in time convergence to zero through the time-integral using the dominated convergence theorem. The second term is small as a function \(L\)  because \(\int_{[L,+\infty)} f_t(\rmd y) = \int_{[L,+\infty)} x^{-1}xf_t(\rmd x) \leq L^{-1}M_1(f_t)\) 
	
	Since \eqref{eq:flux-with-epsilon-and-smooth} holds, and each of the terms is an arbitrarily good approximation of the corresponding limit term, it follows that for all \(\phi \in C_c^1(\R_+)\) and for all \(t \in [0, +\infty)\), we have
	\begin{align}
	\inner{\phi}{xf_t} = \frac{1}{2}\int_{0}^t\int_{\R_*}\int_{\R_*} \left((x+y)\phi(x+y)-x\phi(x)-y\phi(y)\right)f_t(\rmd x)f_t(\rmd y) + t\phi(0).
	\end{align}
	
	To prove the uniqueness of the solution, assumme that we have a \(f \in C([0,T], \posrad{\R_*})\) which satisfies the required moment bounds and satisfies the above evolution equation. We can approximate the function 
	\begin{align}
	\phi_\lambda(x) = \begin{cases}
	\frac{(1-\rme^{-\lambda x})}{x},\quad  &x > 0,\\
	\lambda, \quad &x = 0	
	\end{cases}
	\end{align}
	in the pointwise sense by a sequence of \(C_c^1(\R_+)\) functions \(\eta_{M}(x)\phi_\lambda(x)\), with \(M \geq 2\), where \(\eta_{M}(x)\) is a \(C_c^1(\R_+)\)-smooth cutoff function that vanishes outside \([0,2M]\), is \(1\) on \([0,M]\), and is monotonely decreasing on the interval \([M,2M]\). These can be constructed in such a way that gives also \(\sup_{M\geq 2}\norm{\partial \eta_M}_\infty <+\infty\). Multiplying any \(\phi_\lambda\) with \(\eta_M\) gives a valid test function, so by the weak formulation, we have
	\begin{align}
	&\int_{\R_*} \eta_M(x)(1-\rme^{-\lambda x})f_t(\rmd x) = \frac{1}{2}\int_{0}^t \int_{\R_*}\int_{\R_*}D[(\cdot)\eta_M\phi_\lambda](x,y)f_s(\rmd x)f_s(\rmd y)\rmd s + t\lambda.
	\end{align}
	
	Let \(\psi_{\lambda, M}(x) = \eta_M(x)(1-\rme^{-\lambda x})\), so that \(\psi_{\lambda,M}(x) = x \eta_M(x)\phi_\lambda(x)\) when restricted to \(\R_*\). Since \(M\geq 2\), the absolute value of the function
	\begin{align}
	D[\psi_{\lambda,M}](x,y)= \eta_M(x+y)\left(1-\rme^{-\lambda (x+y)}\right)-\eta_M(x)(1-\rme^{-\lambda x}) - \eta_M(y)(1-\rme^{-\lambda y})
	\end{align}
	is bounded from above by the function \(g \colon \R_*^2 \to \R_+\) which is defined pointwise as
	\begin{align}
	g(x,y) = 
	\begin{cases}
	(1-\rme^{-\lambda x})(1-\rme^{-\lambda y}), \quad &(x,y) \in [0,1]^2 \\
	(1-\rme^{-\lambda x}) + x\sup_{M \geq 2}\norm{\cf{(1,+\infty)}\partial (\eta_M\psi_{\lambda,M})}_\infty,\quad & x\leq 1, y > 1 \\
	(1-\rme^{-\lambda y}) + y \sup_{M \geq 2}\norm{\cf{(1,+\infty)}\partial (\eta_M \psi_{\lambda,M})}_\infty, \quad &y \leq 1, x > 1 \\
	2, \quad &x,y > 1.
	\end{cases}
	\end{align}
	This is an integrable function with respect to the product measure \(f_t \otimes f_t\), since the first moment of \(f_t\) is finite. We may therefore use the Dominated Convergence Theorem to obtain the time-evolution for the Bernstein transform. Indeed, for all \(t \in [0,+\infty)\), we have
	\begin{align}
	\bern{f_t}(\lambda) = -\int_{0}^t \rmd s(\bern{f_s}(\lambda))^2 + t\lambda.
	\end{align}
	The Picard--Lindelöf existence and uniqueness theorem for ODEs provides us with a unique \emph{local solution} to this equation, since the function \(x \mapsto x^2\) is locally Lipschitz in \(\R\). Moreover, this solution lives inside the invariant region
	\begin{align}
	D_\lambda \coloneqq [0,\sqrt{\lambda}],
	\end{align}	
	on which the mapping \(x \mapsto x^2\) is Lipschitz-continuous. We can therefore paste these local solutions together to form a unique global solution. This solution is then the following function:
	\begin{align}
	t \mapsto \sqrt{\lambda}\tanh(\sqrt{\lambda} t).
	\end{align}
	
	Such a procedure works for all \(\lambda \in \R_*\), and thus any solution will have the Bernstein transform that corresponds to our solution. By the Lévy--Khintchine representation of Bernstein functions, the uniqueness of the solution follows.
	\end{proof}

	To the dynamical problem of flux from zero, we have the corresponding problem of finding a stationary solution.
	\begin{definition}
	\label{def:flux-from-zero-stationary}
	A measure \(\mu \in \posrad{\R_*}\) is called a stationary solution to the constant kernel coagulation equation with a flux of particles at zero, in case it meets the following requirements:
	\begin{itemize}
	\item[1.] \(\int_{\R_*}\min(1,x) \mu(\rmd x) <+ \infty\). In other words, the distribution of small particles satisfies \(\int_{(0,1)}x\mu(\rmd x) <+\infty\) and the distribution of large particles satisfies \(\int_{[1,+\infty)} \mu(\rmd x) < +\infty\).
	\item[2.] For all \(\phi \in C^1_c(\R_+)\), we have
	\begin{align}
	\label{eq:constant-kernel-stationary-weak}
	-\iint_{\R_*^2}D[x\phi](x,y)\mu(\rmd x)\mu(\rmd y) = \phi(0).
	\end{align}
	\end{itemize}
	\end{definition}	
	
	\begin{proposition}
	\label{prop:stationary-solution:unique}
	There exists a unique stationary solution to the constant kernel coagulation equation with a flux of particles at zero in the sense of Definition \ref{def:flux-from-zero-stationary}. This measure \(\bar{f} \in \posrad{\R_*}\) absolutely continuous with respect to the Lebesgue measure and is given by 
	\begin{align}
	\label{def:constant-kernel-stationary}
	\bar{f}(\rmd x) = \frac{1}{2\sqrt{\pi}}x^{-3/2} \rmd x
	\end{align}
	\end{proposition}	
	\begin{proof}
		Let \(\eta_M\) be the test function from the proof of Proposition \ref{prop:existence-constant-kernel}. Define \(\phi_\lambda \colon \R_+ \to \R\) as
		\begin{align}
		\phi_\lambda(x) 
			= \begin{cases}
			   \frac{1-\rme^{-\lambda x}}{x}, \quad &x \in \R_*\\
			 \lambda, \quad & x=0
		\end{cases}.
		\end{align} Define functions \(\phi_{\lambda,M}, \psi_{\lambda,M} \colon \R_+ \to \R\) by \(\phi_{\lambda,M}(x) = \eta_M(x) \phi_\lambda(x)\) and \(\psi_{\lambda,M}(x) = x \phi_{\lambda,M}(x)\)
		Now, by \eqref{eq:constant-kernel-stationary-weak}, we have for any solution \(\mu\) that satisfies the stationary coagulation equation with a flux of particles at zero in the sense of Definition \ref{def:flux-from-zero-stationary} that for all \(\lambda > 0\) and all \( M \geq 2\)
		\begin{align}
		-\iint_{\R_*^2} D[\psi_{\lambda,M}](x,y)\mu(\rmd x)\mu(\rmd y) = \lambda.
		\end{align}
		On the other hand, the function 
		\[D[\psi_{\lambda,M}](x,y) = \eta_M(x+y)(1-\rme^{-\lambda (x+y)}) - \eta_M(x)(1-\rme^{-\lambda x})-\eta_M(y)(1-\rme^{-\lambda y})\] satisfies, for all pair of points \(x,y \in \R_*\)
		\begin{align}
		\abs{D[\psi_{\lambda,M}](x,y)} \leq h(x,y),
		\end{align}
		where the function \(h \colon \R_*^2 \to \R_+\) is defined as
		\begin{align}
		h(x,y) &= \begin{cases}
					(1-\rme^{-\lambda x})(1-\rme^{-\lambda y}), \quad &x<1,y<1 \\
					(1-\rme^{-\lambda x}) + x \sup_{M\geq 2}\sup_{z\in [1/2,+\infty)}\abs{\partial_z(\phi_{\lambda,M}(z))}, \quad &x \leq 1, y > 1 \\
					(1-\rme^{-\lambda y}) + y \sup_{M\geq 2}\sup_{z\in [1/2,+\infty)}\abs{\partial_z\phi_{\lambda,M}(z)}, \quad &x > 1, y \leq 1 \\
					2, \quad &x,y>1
	 			  \end{cases}
		\end{align}
		By the moment assumption, this is an integrable function with respect to the product measure \(\mu \otimes \mu\) on \(\R_*^2\). Moreover, for a fixed \(\lambda\) and for all \((x,y) \in \R_*^2\), the function \(D[\psi_{\lambda,M}]\) converges pointwise as \[D[\psi_{\lambda,M}](x,y) \to -(1-\rme^{-\lambda x})(1-\rme^{-\lambda y}), \quad \text{ as } M \to +\infty.\] Therefore, the dominated convergence theorem then gives us that
		\begin{align}
		\label{eq:stationary-bernstein}
		\bern{\mu}(\lambda)^2 = \lambda.
		\end{align}
		The same argument works for any given \(\lambda\), so the identity \eqref{eq:stationary-bernstein} is true for all \(\lambda > 0\). Since \(\mu\) is a positive measure, its Bernstein transform \(\bern{\mu}(\lambda)\) is positive for all \(\lambda > 0\). It therefore follows that the Bernstein transform of \(\mu\) at point \(\lambda\) is given by the square root of \(\lambda\):
		\begin{align}
		\bern{\mu}(\lambda) = \sqrt{\lambda}, \quad \lambda \in \R_*.
		\end{align}
		
		The Levý--Khintchine representation theorem associates this Bernstein function uniquely with a measure \(\bar{f} \in \posrad{\R_*}\). The resulting measure is absolutely continuous with respect to the Lebesgue measure and has the following simple power law form 
		\begin{align}
		\label{eq:constant-kernel-stationary-def}
		\bar{f}(\rmd x) = \frac{1}{2\sqrt{\pi}}x^{-\frac{3}{2}}\rmd x.
		\end{align}
	
		This measure satisfies the following integrability condition
		\begin{align}
		\frac{1}{2\sqrt{\pi}}\int_{\R_*}\min(1,x)x^{-\frac{3}{2}}\rmd x < +\infty.
		\end{align}
		
		However, both the mass and the number of particles is infinite, since the zeroeth and first moments are unbounded.
	\end{proof}		
	
	The following result shows that in the system there is a constant transport of mass from particles of size smaller than \(z\) to particles of size bigger than \(z\). Thus, even though the solution is stationary in the sense that it does not change its shape in time, the nonzero flux shows that the system is not in equilibrium.	
	
	\begin{proposition}
	\label{prop:stationary-constant-flux}
	The stationary measure \(\bar{f}\) in \eqref{eq:constant-kernel-stationary-def} has a constant flux, which means that it satisfies the identity \(J_{\bar{f}}(z) = 1\) for all \(z \in \R_*\).
	\end{proposition}
	\begin{proof}
	The function \((x,y) \mapsto \cf{\Omega_z}2x\bar{f}(x)\bar{f}(y)\) is integrable over the set \(\R_*^2\). By Fubini's theorem, we may therefore compute
	\begin{align}
	\iint_{\Omega_z} 2x \bar{f}(x)\bar{f}(y)\rmd x \rmd y
	= \int_{0}^z \int_{z-x}^\infty 2x\bar{f}(x)\bar{f}(y)\rmd y \rmd x &= \frac{1}{2\pi}\int_{0}^z \int_{z-x}^\infty x^{-1/2} y^{-3/2}\rmd y \rmd x.
	\end{align}
	Since the integral over \(y\) evaluates as \(\int_{z-x}^\infty y^{-3/2} \rmd y = 2(z-x)^{-1/2}\), we obtain the desired identity
	\begin{align}
	\frac{1}{2\pi}\int_{0}^z \int_{z-x}^\infty x^{-1/2}y^{-3/2} \rmd y \rmd x = \frac{1}{\pi} \int_{0}^z x^{-1/2}(z-x)^{-1/2} \rmd x = 1.
	\end{align}
	\end{proof}
	
	Finally, the unique time-dependent solutions converge to the stationary solution in the long time limit. This is the content of the next proposition.
	\begin{proposition}
	\label{prop:constant-kernel-convergence-stationary}
	The unique solution to the constant kernel coagulation equation with a flux from zero, which we denote by \(f\) and which exists by Proposition \ref{prop:existence-constant-kernel}, converges in the \wkst-topology as \(t \to +\infty\) to the unique stationary solution \(\bar{f}\).
	
	\end{proposition}
	\begin{proof}
	The \wkst-convergence of the time-evolved measures \(f_t\) towards the stationary measure \(\bar{f}\) follows directly from the fact that the Bernstein transforms satisfy, for every \(\lambda \in \R_*\), \[\bern{f_t}(\lambda) = \sqrt{\lambda}\tanh(\sqrt{\lambda}t) \to \sqrt{\lambda} = \bern{\bar{f}}(\lambda)\] as \(t \to +\infty\).
	\end{proof}
	
	\begin{remark}
	The measure \(f_t\) satisfies, for all \(\phi \in C_c(\R_*)\) and all times \(t \geq 0\), the identity 
	\begin{align}
	\int_{\R_*} \phi(x)f_t(\rmd x) = \int_0^t \iint_{\R_*^2}D[\phi](x,y)f_s(\rmd x)f_s(\rmd y).
	\end{align}
	
	In other words, the time-dependent measure \(f\) is a solution to the coagulation equation without a source term, if we consider a formulation of the problem where only test functions \(\phi \in C_c(\R_*)\) are allowed. The reason why this does not contradict the uniqueness result of Menon and Pego \citep[Theorem~2.4]{menon_approach_2004} is that the solutions disagree when tested against those functions that include the point \(x=0\) in their support.
	\end{remark}	
	\newpage
	\bibliography{Coagulation}
\end{document}